\documentclass[arxiv,reqno]{ml-gen}

\usepackage{amsmath,amscd,upref} 
\usepackage{amssymb,amsfonts,mathrsfs}
\usepackage{graphicx}
\usepackage{xy} \xyoption{all}
\usepackage[symbol]{footmisc}
\usepackage{perpage}
\MakePerPage[5]{footnote}
\usepackage[colorlinks=true,linkcolor=blue,citecolor=blue,backref=page]{hyperref}
\usepackage{math60,mlthm}
\usepackage{local}

\usepackage{euler}

\begin{document}
\title{A gluing formula for the analytic torsion on singular spaces}

\author{Matthias Lesch}
\address{Mathematisches Institut,
Universit\"at Bonn,
Endenicher Allee 60,
53115 Bonn,
Germany}

\email{ml@matthiaslesch.de, lesch@math.uni-bonn.de}
\urladdr{www.matthiaslesch.de, www.math.uni-bonn.de/people/lesch}
\thanks{Supported by the Hausdorff Center for Mathematics, Bonn}
\dedicatory{To my family}


\begin{abstract} We prove a gluing formula for the analytic torsion
on non-compact (i.e. singular) Riemannian manifolds. 
Let $M= U\cup_{\pl M_1} M_1$, where $M_1$ is a compact manifold with boundary and
$U$ represents a model of the singularity. For general elliptic operators
we formulate a criterion, which can be checked solely on $U$, for the existence of a global heat expansion, in particular for the existence
 of the analytic torsion in case of the Laplace operator.
The main result then is the gluing formula for the analytic torsion.
Here, decompositions $M=M_1\cup_Y M_2$ along any compact closed hypersurface $Y$ with
$M_1, M_2$ both non-compact are allowed; however product structure near $Y$
is assumed. We work with the de Rham complex coupled to an arbitrary flat bundle $F$;
the metric on $F$ is not assumed to be flat. 
In an appendix the corresponding algebraic gluing formula is proved.
As a consequence we obtain a framework for proving a Cheeger-M\"uller type Theorem
for singular manifolds; the latter has been the main motivation
for this work.

The main tool is Vishik's theory of moving boundary value problems for the
de Rham complex which has also been successfully applied to Dirac type
operators and the eta invariant by J. Br\"uning and the author. 
The paper also serves as a new, self--contained, and brief approach to Vishik's
important work.
\end{abstract}

\maketitle
\tableofcontents
\listoffigures

\section{Introduction}\label{s:intro}     
The Cheeger--M\"uller Theorem \cite{Che:ATH, Mul:ATT,Mul:ATR} on the equality
of the analytic and combinatorial torsion is one of the cornerstones
of modern global analysis. To extend the theorem to certain singular manifolds
is an intriguing open challenge.

In his seminal work \cite{Che:SGS, Che:SGSR} Cheeger initiated the program
of ``extending the theory of the Laplace operator to certain Riemannian
spaces with singularities".
Since then a lot of work on this program has been done. It is impossible
to give a proper account here, but let us mention Br\"uning and Seeley
\cite{BruSee:ITF, BruSee:RES}, Melrose and collaborators \cite{Mel:APS},
and Schulze and collaborators \cite{Sch:PDO}.
While the basic spectral theory (index theory, heat kernel analysis)
for several types of singularities (cones \cite{Les:OFT}, cylinders
 \cite{Mel:APS}, cusps \cite{Mul:STR}, edges \cite{Maz:ETD})
is fairly well understood, an analogue of the Cheeger-M\"uller Theorem has not yet
been established for any type of singular manifold 
(except compact manifolds with boundary).

We will not solve this problem in this paper. However, we will provide 
a framework for attacking the problem.  

To describe this we must go back a little. 
Let $M$ be a Riemannian manifold (boundaryless but not necessarily compact,
also the \emph{interior} of a manifold with boundary is allowed) and let $P^0$
be an elliptic differential operator 
acting on the sections $\Gamma^\infty(E)$ of the Hermitian vector bundle $E$. 
We consider $P^0$ as an unbounded operator in the Hilbert space $L^2(M,E)$ of $L^2$--sections of $E$. 
Moreover, we assume $P^0$ to be bounded below; e.g. $P^0=D^tD$ for an elliptic operator $D$.
Fix a bounded below self--adjoint extension $P\ge -C>-\infty$.

$e^{-tP}$ is an integral operator with a smooth kernel $k_t(x,y)$ which on the
diagonal has a pointwise asymptotic expansion
\begin{equation}\label{eq:AsympInterior}
    k_t(x,x)\sim_{t\searrow 0} \sum_{j=0}^\infty a_j(x) \, t^{\frac{j-\dim
    M}{\ord P}}.
\end{equation} 
This asymptotic expansion is \emph{uniform on compact subsets
of $M$} and hence if e.g. $M$ is compact it may be integrated over the
manifold to obtain an asymptotic expansion for the trace of $e^{-tP}$.
For general non-compact $M$ one cannot expect the operator $e^{-tP}$ to be of
trace class. Even if it is of trace class and even if the coefficients
$a_j(x)$ in Eq.\ \eqref{eq:AsympInterior} are integrable, integration of
Eq.\ \eqref{eq:AsympInterior} does not necessarily lead to an asymptotic
expansion of $\Tr\bigl(e^{-tP}\bigr)$. 
It is therefore a fundamental problem to give criteria which ensure that
$e^{-tP}$ is of trace class and such that there is an asymptotic expansion
\begin{equation}\label{eq:AsympGlobal}
 \Tr\bigl(e^{-t P}\bigr)\sim_{t\searrow 0} 
 \sum_{\substack{
  \Re\ga\to  \infty\\
  0\le k\le k(\ga)}
  }
               a_{\ga k}\; t^\ga \log^k t.
\end{equation}
It is not realistic to find such criteria for arbitrary open manifolds.
Instead one looks at geometric differential operators on manifolds with
singular exits which occur in geometry. A rather generic description of
this situation can be given as follows: suppose that there is a compact
manifold $M_1\subset M$ and a \textquotedblleft well
understood\textquotedblright\ 
model manifold $U$ such that
\begin{equation}\label{eq:100a}  
   M= U\cup_{\pl M_1} M_1.   
\end{equation}
We list a couple  of examples for $U$ which are reasonably well understood 
and which are of geometrical significance:

\smallskip
\paragraph*{\itshape{1. Smooth boundary}} $U=(0,\eps)\times Y$ is a
cylinder with metric $dx^2+ g_Y$ over a smooth compact boundaryless manifold $Y$. 
Then $M$ is just the interior of a compact manifold with
boundary. To this situation the theory of elliptic boundary value
problems applies. Heat trace expansions are established, \eg for
all well-posed elliptic boundary value problems associated to
Laplace-type operators \cite{Gru:TEP}.

\paragraph*{\itshape{2. Isolated asymptotically conical singularities} }
 $U=(0,\eps)\times Y$ with metric $dx^2+ x^2 g_Y(x)$. Then
  $M$ is a manifold with an isolated (asymptotically) conical singularity. This
  is the best understood case of a singular manifold; it is impossible
  here to do justice to all the scientists who contributed. So we just
  reiterate that its study was initiated by Cheeger \cite{Che:SGS, Che:SGSR}.

  \paragraph*{\itshape{3. Simple edge singularities}} In the hierarchy of singularities
of stratified spaces, which are in general of iterated cone type, this
is the next simple class after isolated conical ones: simplifying a little
$U$ is of the form $(0,\eps)\times F\times B$ with metric
$dx^2+x^2 g_F(x)+g_B(x)$. The heat trace expansion and the existence of
the analytic torsion for this class of singularities has been established
recently by Mazzeo and Vertman \cite{MazVer:ATM}.

\paragraph*{\itshape{4. Complete cylindrical ends} }
This case is at the
heart of Melrose's celebrated b--calculus \cite{Mel:APS}.
An exact b-metric on $(0,\eps)\times Y$ is of the form
$dx^2/x^2+g_Y$. Making the change of variables $x=e^{-y}$
we obtain a metric cylinder $(-\log\eps,\infty)\times Y$
with metric $dy^2+g_Y$. 
$M$ is then a complete manifold. Therefore, the Laplacian, \eg
is essentially self--adjoint. However, it is not a discrete
operator and hence its heat operator is not of trace class.

\paragraph*{\itshape{5. Cusps}} Cusps occur naturally as singularities
of Riemann surfaces of constant negative curvature. A cusp
is given by $U=(0,\infty)\times Y$ with metric $dx^2+ e^{-2x} g_Y$. Then
$M$ has finite volume. As in the previous case, however, 
the Laplacian is not a discrete operator. 
In this situation (and also in the previous one) 
one employs methods from scattering theory.
There has been seminal work on this by Werner M\"uller \cite{Mul:SGS}.
\smallskip

The results of this paper apply to situations where the operator $P$
is discrete (has compact resolvent). This is the case
in the examples 1.-3. above, but \emph{not} in 4. and 5. Nevertheless
we are confident that our method can be extended to relative heat traces
and relative determinants, \eg for surfaces of finite area. 

\medskip
To explain our results 
without becoming too technical suppose that for $P_U=P\restriction U$ and $P_1=P\restriction M_1$ 
(of course suitable extensions have to be chosen for $P_U$ and $P_1$) we have proved expansions
Eq.\ \eqref{eq:AsympGlobal}. Then in terms of a suitable cut-off function
$\varphi$ which is $1$ in a neighborhood of $M_1$ one expects to hold:

\begin{princ}[Duhamel's principle for heat asymptotics; informal
version]\label{thm:Duhamel}
If $P_U$ and $P_1$ are discrete with trace-class heat kernels then so is $P$
and
\begin{equation}\label{eq:Duhamel}
 \Tr\bigl(e^{-tP}\bigr) =\Tr\bigl( \varphi e^{-t
 P_1}\bigr)+\Tr\bigl((1-\varphi)e^{-tP_U}\bigr)+O(t^N),\quad\text{as } t\to 0+
\end{equation}
for all $N$.
\end{princ}
We reiterate that the heat operator is a \emph{global operator}. On
a \emph{closed} manifold its short time asymptotic expansion is local
in the sense that the heat trace coefficients are integrals over local
densities as described above. This kind of local behavior cannot be
expected on non-compact manifolds. However, Principle \ref{thm:Duhamel}
shows that the heat trace coefficients localize near the singularity; they
may still be global \emph{on} the singularity as it is the case, \eg
for Atiyah-Patodi-Singer boundary conditions \cite{APS:SARI}.
 
Principle \ref{thm:Duhamel} is a folklore theorem which appears in
various versions in the literature. In Section \ref{s:EOG} below we will
prove a fairly general rigorous version of it (Cor. \ref{cor:EOG6}).

Once the asymptotic expansion Eq.\ \eqref{eq:AsympGlobal} is in place one
obtains, via the Mellin transform, the meromorphic continuation of the
$\zeta$--function
\begin{equation}
       \zeta(P;s):= \sum_{\gl\in\spec(P)\setminus\{0\}} \gl^{-s}
                = \frac{1}{\Gamma(s)}\int_0^\infty t^{s-1} \Tr\bigl( (I-\Pi_{\ker P}) e^{-t P}
                                                \bigr) dt.
\end{equation}

Let us specialize to the de Rham complex. So suppose that we have chosen an 
ideal boundary condition
(essentially this means that we have chosen closed extensions for the exterior
derivative) $(\sD,D)$ for the de Rham complex such that the corresponding
extensions $\Delta_j=D_j^*D_j+D_{j-1}D_{j-1}^*$ of the Laplace operators
satisfy Eq.\ \eqref{eq:AsympGlobal}. Then we can form the analytic torsion
of $(\sD,D)$
\begin{equation}
 \log T(\sD,D):=\frac 12 \sum_{j\ge 0} (-1)^j j\frac{d}{ds}\big|_{s=0} 
\zeta(\Delta_j;s).
\end{equation}

For a closed manifold the celebrated Cheeger--M\"uller Theorem
(\cite{Che:ATH},\cite{Mul:ATT}) relates the analytic torsion to the
combinatorial torsion (Reidemeister torsion).

In terms of the decomposition Eq.\ \eqref{eq:100a} the problem
of proving a CM type Theorem for the singular manifold $M$
decomposes into the following steps.

\begin{enumerate}
\item Prove that the analytic torsion exists for the model manifold $U$.
\item Compare the analytic torsion with a suitable combinatorial
 torsion for $U$.
\item Prove a gluing formula for the analytic and combinatorial torsion
 and apply the known Cheeger--M\"uller Theorem for the manifold with boundary
 $M_1$.
\end{enumerate}

A gluing formula for the combinatorial torsion is more or less an algebraic
fact due to Milnor; \cf also the Appendix \ref{s:HAG}. 
The following Theorem which follows from our gluing formula
solves (3) under a product structure assumption:

\begin{theorem}\label{t:I1}
Let $M$ be a singular manifold as \Eqref{eq:100a} and assume that
near $\pl M_1$ all structures are product. Then for 
establishing a Cheeger-M\"uller Theorem for $M$
it suffices to prove it for the model space $U$ of the singularity.
\end{theorem}
The Theorem basically says that, under product assumptions, one gets step (3) for free. 
Otherwise the specific form of $U$ is completely irrelevant. We conjecture
that the product assumption in Theorem \ref{t:I1} can be dispensed with. 
This would follow once the anomaly formula of Br\"uning-Ma \cite{BruMa:AFR}
were established for the model $U$ of the singularity; this would allow to
compare the analytic torsion for $(U,g)$ to the torsion
of $(U,g_1)$, where $g_1$ is product near $\pl M_1$ and outside a relatively
compact collar coincides with $g$.

The Theorem is less obvious than it sounds since torsion invariants are global in nature. 
However, we will show here that under minimal technical assumptions
the analytic torsion satisfies a gluing formula. That the combinatorial
torsion satisfies a gluing formula is a purely algebraic fact (\textit{cf.}
Appendix \ref{s:HAG}).
The blueprint for our proof is a technique of moving boundary conditions due to Vishik 
\cite{Vis:GRS} who applied it to prove the Cheeger-M\"uller 
Theorem for compact manifolds with smooth boundary.
Br\"uning and the author \cite{BruLes:EIC} applied Vishik's moving boundary
conditions to generalized Atiyah-Patodi-Singer nonlocal boundary conditions
and to give an alternative proof of the gluing formula for the eta-invariant.
We emphasize, however, that the technical part of the present paper
is completely independent of (and in our slightly biased view simpler than)
\cite{Vis:GRS}. Also we work with the de Rham complex coupled to an
arbitrary flat bundle $F$. Besides the product structure assumption we do
not impose any restrictions on the metric $h^F$ on $F$; in particular $h^F$
is not assumed to be flat.

We note here that in the context of \emph{closed} manifolds gluing formulas
for the analytic torsion have been proved in \cite{Vis:GRS},
\cite{BurFriKap:TMB}, and recently \cite{BruMa:GFA}.
In contrast our method applies to a wide class of singular manifolds.

Some more comments on conic singularities, the most basic singularities,
are in order: let $(N,g)$ be a compact closed Riemannian manifold and let 
$CN=(0,1)\times N$ with metric $dx^2+x^2 g$ be the cone over $N$. We emphasize
that sadly near $\pl CN=\{1\}\times N$ we do not have product
structure. Let $g_1$ be a metric on $CN$ which is product near $\{1\}\times N$
and which coincides with $g$ near the cone tip.

Vertman \cite{Ver:ATB} gave formulas for the torsion of the cone $(CN,g)$ in
terms of spectral data of the cone base. What is still not yet understood is
how these formulas for the analytic torsion can be related to a combinatorial
torsion of the cone, at least not in the interesting odd dimensional case.
For $CN$ even dimensional Hartmann and Spreafico \cite{HarSpr:ECM} 
express the torsion of $(CN,g)$ in terms of the intersection torsion
introduced by A. Dar \cite{Dar:IRT} and the anomaly term of Br\"uning-Ma
\cite{BruMa:AFR}. If it were also possible to apply loc. cit. to the singular
manifold $CN$ to compare the torsion of the metric cone $(CN, g)$ to that of
the cone $(CN,g_1)$ where the metric near $\{1\}\times N$ is modified to a product
metric then one would obtain a (very sophisticated) new proof of Dar's
Theorem that for an even dimensional manifold with conical singularities
the analytic and the intersection torsion both vanish\footnote{
 For this to hold one needs to assume that the metric on the twisting bundle $F$ 
 is also flat.}.
It would be more interesting, of course, to have this program worked out
in the odd dimensional case.

\medskip
The paper is organized as follows. Section \ref{s:OMZ} serves to
introduce some terminology and notation. In a purely functional
analytic context we discuss selfadjoint operators with \emph{discrete
dimension spectrum}; this terminology is borrowed from Connes and Moscovici's
celebrated paper on the Local Index Theorem in Noncommutative Geometry \cite{ConMos:LIF}. 
For \emph{Hilbert complexes} \cite{BruLes:HC}
whose Laplacians have discrete dimension spectrum one can introduce
the analytic torsion. We state a formula for the torsion of
a product complex (Prop. \ref{p:1007116}) and in Subsection
\plref{ss:TFD} we collect some algebraic facts about determinants
and the torsion of a finite--dimensional Hilbert complex. The
main result of the Section is Prop. \ref{p:GenDiffResult}
which, under appropriate assumptions, provides a variation
formula for the analytic torsion of a one-parameter family of Hilbert
complexes.

In Section \plref{s:EOG} we discuss the gluing of operators in a fairly
general setting: we assume that we have two pairs $(M_j,P_j^0), j=1,2$
consisting of Riemannian manifolds $M_j^m$ and elliptic operators $P_j^0$
such that each $M_j$ is the interior of a manifold
$\ovl{M_j}$ with compact boundary $Y$ ($\ovl{M_j}$
is not necessarily compact). Let $W=Y\times (-c,c)$ be a common collar
of $Y$ in $M_1$ resp. $M_2$ such that $\pl M_1=Y\times \{1\}$ and 
$\pl M_2=Y\times \{-1\}$ and such that $P_1^0$ coincides with $P_2^0$ 
over $W$.  Then $P_j^0$ give rise naturally to a differential operator
$P^0=P_1^0\cup P_2^0$ on $M:=\bigl(M_1\setminus (Y\times (0,c))\bigr)
\cup_{Y\times \{0\}} \bigl(M_2\setminus (Y\times (-c,0))\bigr)$.
Without becoming too technical here we will show 
in Prop. \ref{p:EOG5} that certain semibounded symmetric extensions $P_j, j=1,2$ of 
$P_j^0$ satisfying a non-interaction condition \eqref{eq:EOG7}
give rise naturally to a semibounded selfadjoint extension of $P^0$.
Furthermore, if $P_j$ have discrete dimension spectrum outside $W$
(\textit{cf.~}the paragraph before Cor. \ref{cor:EOG6})
then the operator $P$ has discrete dimension spectrum and up to an
error of order $O(t^\infty)$ the short time heat trace expansion of
$P$ can be calculated easily from the corresponding expansions of $P_j$.

As an additional feature we prove similar results for perturbed operators
of the form $P_j+V_j$ where $V_j$ is a certain non-pseudodifferential
operator; such operators will occur naturally in our main technical Section
\plref{s:GTP}. 

In Section \ref{s:VMB} we describe the details of the gluing
situation, review Vishik's moving boundary conditions for the de
Rham complex in this context, and introduce various one-parameter families
of de Rham complexes. The main technical result of the paper
is Theorem \ref{thm:Theorem2} which analyzes the variation of the
torsions of these various families of de Rham complexes. 
The proof of Theorem \ref{thm:Theorem2} occupies the
whole Section \ref{s:GTP}. The proof is completely independent
of Vishik's original approach. The main feature of our proof
is a gauge transformation \`a la Witten of the de Rham complex which
transforms the de Rham operator, originally a family of operators with
varying domains, onto a family of operators with constant domain;
this family can then easily be differentiated by the parameter.

Theorem \ref{thm:Theorem1} in Section \ref{s:GF} then finally is
the main result of the paper whose proof, thanks to Theorem \ref{thm:Theorem2}
is now more or less an exercise in diagram chasing. 

Appendix \ref{s:HAG} contains the analogues of our main results
for finite-dimensional Hilbert complexes.

\medskip
The paper has a somewhat lengthy history. The material of Sections
\ref{s:VMB} and \ref{s:GTP}, however only in the context of smooth
manifolds, was developed in summer 1999 while being on a Heisenberg
fellowship in Bonn. In light of the (negative) feedback received at
conferences I felt that the subject was dying and therefore abandoned it.

In recent years there has been a revived interest in generalizing
the Cheeger-M\"uller Theorem to manifolds with singularities 
(\cite{MazVer:ATM}, \cite{Ver:ATB}, \cite{MulVer:MAA}, \cite{HarSpr:ECM}). 
I noticed that my techniques (an adaption of Vishik's
work \cite{Vis:GRS} plus simple observations based on Duhamel's principle)
do not require the manifold to be closed. The bare minimal assumptions
required for the analytic torsion to exist (``discrete dimension spectrum''
see Section \ref{s:OMZ})
and a mild but obvious non-interaction restriction on the choice of
the ideal boundary conditions (Def. \ref{def:EOG4})
for the de Rham complex actually suffice to prove a gluing formula
for the analytic torsion. Since a more concise and more accessible account
of Vishik's important long paper \cite{Vis:GRS} is overdue anyway I therefore
eventually, also because Werner M\"uller and Boris Vertman have been pushing
me for quite a while, to make a final effort to write up this paper.


\section*{Acknowledgments} 

I would like to thank Werner M\"uller and Boris Vertman for pushing
and encouraging me to complete this project. I also owe a lot
of gratitude to my family. I dedicate this paper to my dearly beloved
late aunt Annels Roth (1923 -- 2012).

\section{Operators with meromorphic $\zeta$--function}\label{s:OMZ} 

Let $\sH$ be a separable complex Hilbert space, $T$ a non-negative selfadjoint operator in $\sH$ with
$p$--summable resolvent for some $1\le p<\infty$. The summability condition
implies that $T$ is a \emph{discrete} operator, \ie
the spectrum of $T$ consists of eigenvalues of finite
multiplicity with $+\infty$ being the only accumulation point. 
Moreover, 
\begin{align}\label{eq:1209071}  
    \Tr\bigl(e^{-t T}\bigr) &= \sum_{\gl\in\spec T} e^{-t\gl}= \dim\ker T 
       + O(e^{-t\gl_1}),\quad\text{as } t\to \infty,
\intertext{and}    
    \Tr\bigl(e^{-t T}\bigr) &= O(t^{-p}),\quad \text{as } t\to 0+.
\end{align}
Here $\gl_1:=\min\bigl(\spec T\setminus \{0\}\bigr)$ denotes the smallest
non-zero eigenvalue of $T$.

As a consequence, the $\zeta$--function
\begin{equation}\label{eq:ZetaFunction}
       \zeta(T;s):= \sum_{\gl\in\spec(T)\setminus\{0\}} \gl^{-s}
                 = \frac{1}{\Gamma(s)}\int_0^\infty t^{s-1} \Tr\bigl( (I-P_{\ker T}) e^{-t T}
                                                \bigr) dt,
\end{equation}
is a holomorphic function in the half plane $\Re s> p$; $P_{\ker T}$ denotes the orthogonal projection onto $\ker T$.

\begin{dfn}\label{def:FDS} Following \cite{ConMos:LIF} 
we say that $T$ has \emph{discrete dimension spectrum} if $\zeta(T;s)$ 
extends meromorphically to the complex plane $\C$ such that on finite vertical strips 
$|\Gamma(s) \zeta(T;s)|=O(|s|^{-N})$, $|\Im s|\to \infty$, for each $N$.
Denote by $\Sigma(T)$ the set of poles of the function $\Gamma(s)\zeta(T;s)$.
\end{dfn}

It then follows that for fixed real numbers $a<b$ there are only finitely many poles in the
strip $a<\Re s<b$. Moreover, as explained e.g. in \cite[Sec.~2]{BruLes:EIC}, the discrete dimension
spectrum condition is equivalent to the existence of an asymptotic expansion
\begin{equation}\label{eq:HeatExpansion}
    \Tr\bigl(e^{-t T}\bigr)\sim_{t\to 0+} \sum_{\substack{\ga \in -\Sigma,\\
    0\le k\le k(\ga)} }
               a_{\ga k}\; t^\ga \log^k t.
\end{equation}
Furthermore, there is the following simple relation between the coefficients of the asymptotic expansion
and the principal parts of the Laurent expansion at the poles of $\Gamma(s)\zeta(T;s)$:
\begin{equation}\label{eq:PolesRelExpansion}
 \Gamma(s)\zeta(T;s)\sim \sum_{\substack{\ga \in -\Sigma,\\ 0\le k\le k(\ga)}} \frac{a_{\ga k}(-1)^k k!}{(s+\ga)^{k+1}}
                   -\frac{\dim\ker T}{s}.
\end{equation}
\subsection{Hilbert complexes and the analytic torsion} \label{ss:HCAT}
We use the convenient language of Hilbert complexes as outlined in \cite{BruLes:HC}. 
Recall that a Hilbert complex
$(\sD,D)$ consists of a sequence of Hilbert spaces $H_j,0\le j\le N$, together with
closed operators $D_j$ mapping a dense linear subspace $\sD_j\subset H_j$ into $H_{j+1}$.
The complex property means that actually $\ran D_j\subset \sD_{j+1}$ and 
$D_{j+1}\circ D_j=0$. We say that a Hilbert complex has discrete dimension spectrum
if all its Laplace operators $\Delta_j=D_j^*D_j+D_{j-1}D_{j-1}^*$ do have
discrete dimension
spectrum in the sense of Def. \ref{def:FDS}.
Note that since $\Delta_j$ has compact resolvent, $(\sD,D)$ is automatically
a Fredholm complex, cf. \cite[Thm. 2.4]{BruLes:HC}. 
For a Hilbert complex $(\sD,D)$ which is Fredholm the finite-dimensional 
cohomology group $H^j(\sD,D)= \ker D_j/\ran D_{j-1}$
is the quotient space of the Hilbert space $\ker D_j$ by the closed subspace $\ran D_{j-1}$
and therefore is naturally equipped with a Hilbert space structure. From the Hodge
decomposition \cite[Cor. 2.5]{BruLes:HC}
\begin{equation}\label{eq:HD}
\begin{split}
     H_j&= \ker D_j \cap \ker D_{j-1}^* \oplus \ran D_{j-1}\oplus \ran D_j^*\\
        &= \ker \Delta_j \oplus \ran D_{j-1}\oplus \ran D_j^*
        \end{split}
\end{equation}
one then sees that the natural isomorphism
$\hat H^j(\sD,D):=\ker \Delta_j=\ker D_j\cap \ker D_{j-1}^*\rightarrow H^j(\sD,D)$ is an isometric
isomorphism. We will always tacitly assume that the cohomology groups are
equipped with this natural Hilbert space structure.

Recall the \emph{Euler characteristic}
\begin{equation}\label{eq:EulerCharacteristic}
    \chi(\sD,D):=\sum_{j\ge 0} (-1)^j \dim H^j(\sD,D)=\sum_{j\ge 0} (-1)^j
    \dim\ker \Delta_j.
\end{equation}
The discrete dimension spectrum assumption implies the validity of the
\emph{McKean--Singer formula}
\begin{equation}\label{eq:McKS}
 \chi(\sD,D)=\sum_{j\ge 0} (-1)^j\, \Tr\bigl(e^{-t\Delta_j}\bigr), \quad \text{for } t>0.
\end{equation}

\begin{dfn}\label{def:AnalyticTorsion} Let $(\sD,D)$ be a Hilbert complex with discrete dimension
spectrum. The \emph{analytic torsion} of $(\sD,D)$ is defined by
\[ \log T(\sD,D):=\frac 12 \sum_{j\ge 0} (-1)^j j\frac{d}{ds}\big|_{s=0} 
\zeta(\Delta_j;s).\]
If $\zeta(\Delta_j;s)$ has a pole at $s=0$ then by $\frac{d}{ds}\big|_{s=0} \zeta(\Delta_j;s)$
we understand the coefficient of $s$ in the Laurent expansion at $0$.

Obviously $\log T(\sD,D)$ can be defined under the weaker assumption that the
function 
\begin{equation}\label{eq:TorsionFunction}
          F(\sD,D;s):=\frac 12 \sum_{j\ge 0} (-1)^j j \zeta(\Delta_j;s)
\end{equation}
extends meromorphically to $\C$.
\end{dfn}
The analytic torsion can also be expressed in terms of the \emph{closed}
resp. \emph{coclosed} Laplacians: put
\begin{align}
 \Delcl{j}&:= \Delta_j\restriction \ran D_{j-1} = D_{j-1} D_{j-1}^*
 \restriction \ran D_{j-1}, \label{eq:cl}\\
 \Delccl{j}&:= \Delta_j\restriction \ran D_{j}^* = D_{j}^* D_{j}
 \restriction \ran D_{j}^*. \label{eq:ccl}
\end{align}
Note that by definition $\Delccl{0}=0$ and $\Delcl{N}=0$  
act on the trivial Hilbert space $\{0\}$; recall that $N$ is the length of
the Hilbert complex.
By the Hodge decomposition  \eqref{eq:HD} the operators $\Delcl{j}$ and
$\Delccl{j}$ are invertible. Moreover,
\begin{equation}\label{eq:clccl}
 \Delcl{j+1}D_j\restriction\ran D_j^* = D_j \Delccl{j}.
\end{equation}
Hence the eigenvalues
of $\Delccl{j}$ and $\Delcl{j+1}$ coincide including multiplicities.
Putting for the moment $A_j:=\Tr\bigl(e^{-t\Delcl{j}}\bigr)=
\Tr\bigl(e^{-t\Delccl{j-1}}\bigr)$ for $j\ge 1$ and $A_0:=0$
we therefore have
\begin{equation}\label{eq:1209102}  
    \Tr\bigl(e^{-t\Delta_j}\bigr)-\dim H^j(\sD,D) = 
    \Tr\bigl(e^{-t\Delcl{j}}\bigr) + \Tr\bigl(e^{-t\Delccl{j}}\bigr)
    =A_j+A_{j+1},
\end{equation}
and hence
\begin{align}
  \sum_{j\ge 0}& (-1)^j\, j \, \Bigl(\Tr\bigl(e^{-t\Delta_j}\bigr)-\dim
  H^j(\sD,D)\Bigr)\nonumber \\
  &= \sum_{j\ge 0} (-1)^j\, j \, (A_j+A_{j+1}) = \sum_{j\ge 0} (-1)^j\, j\, A_j
  -\sum_{j\ge 0} (-1)^j\, (j-1)\, A_j\nonumber\\
  &  = \sum_{j\ge 0} (-1)^j \Tr\bigl(e^{-t\Delcl{j}}\bigr)= - \sum_{j\ge 0} (-1)^j
    \Tr\bigl(e^{-t\Delccl{j}}\bigr).\label{eq:HTccl}
\end{align} 
To avoid cumbersome distinction of cases we understand that
$\Tr\bigl(e^{-t\Delccl{0}}\bigr)=0$.

\begin{prop}\label{p:1007116}
Let $(\sD',D'),(\sD'',D'')$ be two Hilbert complexes with discrete dimension spectrum.
Let $(\sD,D):=(\sD',D')\hat\otimes(\sD'',D'')$ be their tensor product. Denote
by $\Delta', \Delta'', \Delta$ the Laplacians of $(\sD',D'), (\sD'',D''),
(\sD,D)$, resp.

Then the function $F(\sD,D;s):=\frac 12 \sum_{j\ge 0} (-1)^j j \zeta(\Delta_j;s)$
extends meromorphically to $\C$. More precisely, in terms of the corresponding function
for the complexes $(\sD',D'), (\sD'',D'')$ we have the equations
\begin{align}
 \chi(\sD,D)&= \chi(\sD',D')\cds \chi(\sD'',D''),\label{eq:OMZ6}\\
   F(\sD,D;s)&=\chi(\sD',D') \cds F(\sD'',D'';s)+\chi(\sD'',D'')\cds
   F(\sD',D';s),\label{eq:OMZ7}\\
\intertext{in particular}
   \log T(\sD,D)&=\chi(\sD',D')\cds \log T(\sD'',D'')+\chi(\sD'',D'')\cds \log
   T(\sD',D').\label{eq:OMZ8}
\end{align}
\end{prop}
\begin{proof} This is an elementary calculation, \cf \cite[Prop. 2.1]{Vis:GRS}
 and \cite[Thm.~2.5]{RaySin:RTL}.
Since
\[\Delta_k=\bigoplus_{i+j=k} \Delta_i'\otimes I+ I\otimes \Delta_j'',\]
we have
\begin{equation}
 \ker(\Delta_k-\gl)=\bigoplus_{\gl'+\gl''=\gl}\bigoplus_{i+j=k}
   \ker(\Delta_i'-\gl')\otimes\ker(\Delta_j''-\gl'').
\end{equation}
This proves \Eqref{eq:OMZ6}, which follows also from the K\"unneth--Theorem
for Hilbert complexes \cite[Cor.~2.15]{BruLes:HC}.
Furthermore,
\begin{equation}\begin{split}
  \sum_{k\ge 0}& (-1)^k k\Tr\bigl( e^{-t\Delta_k}\bigr)
  =\sum_{k\ge 0} (-1)^k k\sum_{i+j=k}\sum_{\substack{
                      \gl'\in\spec\Delta_i' \\ 
                      \gl''\in\spec \Delta_j''}}
                      e^{-t\gl'}\, e^{-t\gl''}\\
 =&\sum_{i,j\ge 0} (-1)^{i+j}(i+j)   \sum_{\substack{
                    \gl'\in\spec\Delta_i'  \\ 
                    \gl''\in\spec \Delta_j''}}
                    e^{-t\gl'}e^{-t\gl''}\\
 =&\bigl(\sum_{i\ge 0} (-1)^i \Tr\bigl(e^{-t\Delta_i'}\bigr)\bigr)\cds \bigl(\sum_{j\ge 0}
(-1)^j j\Tr\bigl(e^{-t\Delta_j''}\bigr)\bigr)\\
  &+\bigl(\sum_{j\ge 0} (-1)^j \Tr\bigl(e^{-t\Delta_j''}\bigr)\bigr)\cds \bigl(\sum_{i\ge 0}
(-1)^i i\Tr\bigl(e^{-t\Delta_i'}\bigr)\bigr).
		\end{split}
\end{equation}
The claim now follows from \Eqref{eq:ZetaFunction} and
the McKean--Singer formula \Eqref{eq:McKS} applied to $\Delta_i',\Delta_j''$.
\end{proof}

Next we state an abstract differentiability result, \cf
\cite[Appendix]{DaiFre:EID}, \cite[Appendix D]{Boh:RIF}:

\begin{prop}\label{p:GenDiffResult}
Let $(\sD^\TT,D^\TT), \TT\in J\subset\R$,
be a one parameter family of Hilbert complexes with
discrete dimension spectrum; let
$\Delta_j^\TT=(D_j^\TT)^*D_j^\TT+D_{j-1}^\TT(D_{j-1}^\TT)^*$ be
the corresponding Laplacians. Assume that
\begin{enumerate}
\item $H_T(\sD^\TT,D^\TT)(t)=\altsum j \Tr\bigl(e^{-t\Delta_j^\TT}\bigr)$
is differentiable in $(t,\TT)\in (0,\infty)\times J$ and 
\begin{equation}\label{eq:106}  
    \frac{d}{d\TT} H_T(\sD^\TT,D^\TT)(t)= t\frac{d}{dt} \Tr\bigl(P e^{-t\Delta^\TT}\bigr)
\end{equation}
with some operator $P$ in $H=\oplus_{j\ge 0} H_j$ with
$P(I+\Delta^\TT)^{-N}$ bounded for some $N$.

\item  $\Delta^\TT$ is a graph smooth family of selfadjoint operators
with constant domain and $\dim \ker \Delta^\TT$ independent of $\TT$.

\item There is an asymptotic expansion
\begin{equation}\label{eq:107}  
    \Tr\bigl(P e^{-t \Delta^\TT}\bigr)\sim_{t\to 0+} \sum_{\ga \in -\Sigma, 0\le k\le k(\ga) }
               a_{\ga k}^\TT\; t^\ga \log^k t
\end{equation}
which is locally uniformly in $\TT$ and 
with $a_{\ga k}^\TT$ depending smoothly on $\TT$.
\item $a_{0k}^\TT=0$ for $k>0$, that is in the asymptotic expansion
\eqref{eq:107} there are no terms of the form $t^0 \log^k t$ for $k>0$.
\end{enumerate}
Then $\TT\mapsto \log T(\sD^\TT,D^\TT)$ is differentiable and
\begin{equation}\label{eq:108}  
 \begin{split}
    \frac{d}{d\TT} &\log T(\sD^\TT,D^\TT)\\
       & = -\frac 12 \LIM_{t\to 0+} \Tr\bigl(P e^{-t \Delta^\TT}\bigr) 
           +\frac 12 \LIM_{t\to\infty} \Tr\bigl(P e^{-t \Delta^\TT}\bigr) \\
       & = -\frac 12 a_{00}^\TT +\frac 12 \Tr\bigl( P\restriction \ker
       \Delta^\TT\bigr).
      \end{split}      
\end{equation}
\end{prop}
Here $\LIM\limits_{t\to a}$ stands, as usual, for the constant term in the
asymptotic expansion as $t\to a$. In (1) we have used the abbreviation $\Delta^\TT:=\bigoplus_{j\ge 0} \Delta_j^\TT$.
\begin{proof}
(2) and (3) guarantee that in the following we may interchange differentiation
by $s$ and by $\TT$:
\begin{equation}\label{eq:109}   
  \begin{split}
  2 \frac{d}{d\TT}  \log T&(\sD^\TT,D^\TT) \\
      &=\frac{d}{d\TT} \frac{d}{ds}_{\big| s=0} \frac{1}{\Gamma(s)}
           \int_0^\infty t^{s-1} \altsum j \Tr\bigl(e^{-t\Delta_j^\TT}-P_{\ker \Delta_j^\TT}\bigr) dt\\
   &=\frac{d}{ds}_{\big| s=0} \frac{1}{\Gamma(s)}
           \int_0^\infty t^{s} \frac{d}{dt} \Tr\bigl(P e^{-t\Delta^\TT}\bigr) dt\\
   &=-\frac{d}{ds}_{\big| s=0} \frac{s}{\Gamma(s)}
           \int_0^\infty t^{s-1} \Tr\bigl(P e^{-t\Delta^\TT}\bigr) dt\\
   &= -\frac{d}{ds}_{\big| s=0} \frac{s}{\Gamma(s)}\Bigl[\bigl(
       \frac{a_{00}^\TT}{s}+c_0^\TT+c_1^\TT s+\ldots\bigr) -
       \frac{\Tr\bigl(P\restriction \ker \Delta^\TT\bigr)}{s}\Bigr]\\
       &=- a_{00}^\TT + \Tr\bigl(P\restriction \ker
       \Delta^\TT\bigr).
  \end{split}
\end{equation}
Assumption (1) was used in the second equality and assumptions (3), (4) were used in the penultimate equality. 
Without the assumption (4) the higher derivatives of the function $1/\Gamma(s)$ at $s=0$
would cause additional terms.
Assumption (2) guarantees in particular that $\Tr\bigl(P_{\ker
\Delta_j^\TT}\bigr)$ is
independent of $\TT$. 
\end{proof}

\subsection{Torsion of a finite-dimensional Hilbert complex}\label{ss:TFD}

This Subsection mainly serves the purpose of fixing some notation.
Let $H_1,H_2$ be finite-dimensional Hilbert spaces. For
a linear map $T:H_1\to H_2$ we put 
\begin{equation}\label{eq:Det}
 \Det(T):=\det(T^*T)^{1/2}.
\end{equation}

If $T:H_1\to H_2, S:H_2\to H_3$ are linear maps then obviously
$\Det(TS)=\Det(T) \Det (S)$. Furthermore, given orthogonal decompositions 
$H_j=H_j^{(1)}\oplus H_j^{(2)}, j=1,2$, such that with respect to these 
decompositions we have
\begin{equation}\label{eq:OMZ13}
  T=\begin{pmatrix} T_1& T_{12}\\ 0 & T_2\end{pmatrix},
\end{equation}
then $\Det(T)=\Det(T_1)\Det (T_2)$.

Let $0\to C^0\stackrel{d_0}{\longrightarrow}C^1
\stackrel{d_1}{\longrightarrow}\ldots\stackrel{d_{n-1}}{\longrightarrow}C^n
\longrightarrow 0$ be a finite-dimensional Hilbert complex. Then the torsion
of this complex satisfies
\begin{equation}\label{eq:OMZ14}
   \log T(C^*,d)=\sum_{p\ge 0} (-1)^p \log \Det(d_p:\ker d_p^\perp
\to \im d_p)=: \log \tau(C^*,d).
\end{equation}
Needless to say each finite-dimensional Hilbert complex is automatically a 
Hilbert complex with discrete dimension spectrum. In fact, since the
zeta--function is entire in this case, for the Laplacian of the complex
the set $\Sigma(\Delta)$ defined in Definition \ref{def:FDS} then equals
the set of poles of the $\Gamma$--function, $\{0, -1 , -2, \ldots \}$.

The following two standard results about the torsion and the determinant
will be needed at several places. The first one is elementary, the second
one due to Milnor \cite{Mil:WT}.

\begin{lemma}\label{l:OMZ5} Let $(C^*_k,d^k), k=1,2$, be finite-dimensional
 Hilbert complexes and $\ga:(C_1^*,d^1)\to (C_2^*,d^2)$ be a chain
 isomorphism. Then
 \begin{equation}
  \begin{split}
  \log \tau(C^*_1,d^1)=& \log \tau(C^*_2,d^2) 
  + \sum_{j\ge 0} (-1)^j \log \Det \bigl(\ga_j:C^j_1 \to C^j_2\bigr)\\
    & - \sum_{j\ge 0} (-1)^j \log\Det \bigl( \ga_{j,*}: H^j(C^*_1,d^1)\to
                           H^j(C^*_2,d^2)\bigr).
                          \end{split}
 \end{equation}
\end{lemma}
\begin{proof} For complexes of length $2$ the formula follows directly
 from \Eqref{eq:OMZ13}. Then one proceeds by induction on the length
 of the complexes $C_1, C_2$. We omit the elementary but a little
 tedious details.
\end{proof} 

\begin{prop}[{\cite[Thm. 3.1/3.2]{Mil:WT}}]\label{p:OMZ6} Let $0\to
 C_1\stackrel{\ga}{\longrightarrow} C\stackrel{\gb}{\longrightarrow} C_2\to 0$
 be an exact sequence of finite-dimensional Hilbert complexes and let
 \begin{equation}
  \sH: 0\to H^0(C_1)\stackrel{\ga_*}{\longrightarrow}
  H^0(C)\stackrel{\gb_*}{\longrightarrow}
  H^0(C_2)\stackrel{\delta}{\longrightarrow} H^1(C_1)\longrightarrow\ldots
 \end{equation}
 be their long exact cohomology sequence. Then
 \begin{multline}\label{eq:OMZ21}
  \log \tau(C^*,d)= \log \tau(C_1^*,d^1)+\log \tau(C_2^*,d^2)+\log\tau(\sH)\\
  -\sum_{j\ge 0} (-1)^j \log\tau\bigl(0\to C_1^j\stackrel{\ga}{\rightarrow}
  C^j\stackrel{\gb}{\rightarrow}C_2^j\to 0\bigr).
 \end{multline}
\end{prop}
In fact the Proposition as stated is a combination of \cite[Thm.~3.2]{Mil:WT} 
and the previous Lemma \plref{l:OMZ5}. The last term in \Eqref{eq:OMZ21}
does not appear in \cite[Thm.~3.2]{Mil:WT} since there one is given
\emph{preferred} bases of $C_1, C, C_2$ which are \emph{compatible}.
In our Hilbert complex setting the preferred bases are the orthonormal
ones. The last term in \Eqref{eq:OMZ21} makes up for the fact that
in general it is not possible to choose
orthonormal bases of $C_1, C, C_2$ which are compatible in the sense of loc.
cit. For a proof in the more general von Neumann setting see
\cite[Theorem 1.14]{BurFriKap:TMB}.

For future reference we note that for the acyclic complex
$(0\to C_1^j\stackrel{\ga}{\rightarrow} C^j\stackrel{\gb}{\rightarrow}C_2^j\to 0\bigr)$
of length 2 on the right of \Eqref{eq:OMZ21} it follows from the definition \Eqref{eq:OMZ14} that
\begin{multline}\label{eq:OMZ26} 
  \log\tau\bigl(0\to C_1^j\stackrel{\ga}{\rightarrow}
  C^j\stackrel{\gb}{\rightarrow}C_2^j\to 0\bigr)\\
  =\frac 12\log \Det(C_1^j\stackrel{\ga^*\ga}{\longrightarrow} C_1^j)
  - \frac 12 \log \Det(C_2^j\stackrel{\gb\gb^*}{\longrightarrow} C_2^j).
 \end{multline}  

 Finally, we remind the reader of the (trivial) fact that if in Prop. \ref{p:OMZ6}
 the complex $C$ equals $C_1\oplus C_2$, $\ga$ the inclusion and $\gb$
 the projection onto the second summand then $\log\tau(\sH)=0$
 and $\log \tau(C^*,d)= \log \tau(C_1^*,d^1)+\log \tau(C_2^*,d^2)$.


\section{Elementary operator gluing and heat kernel estimates on non-compact
manifolds}\label{s:EOG} 
\begin{figure}
\includegraphics[width=10cm]{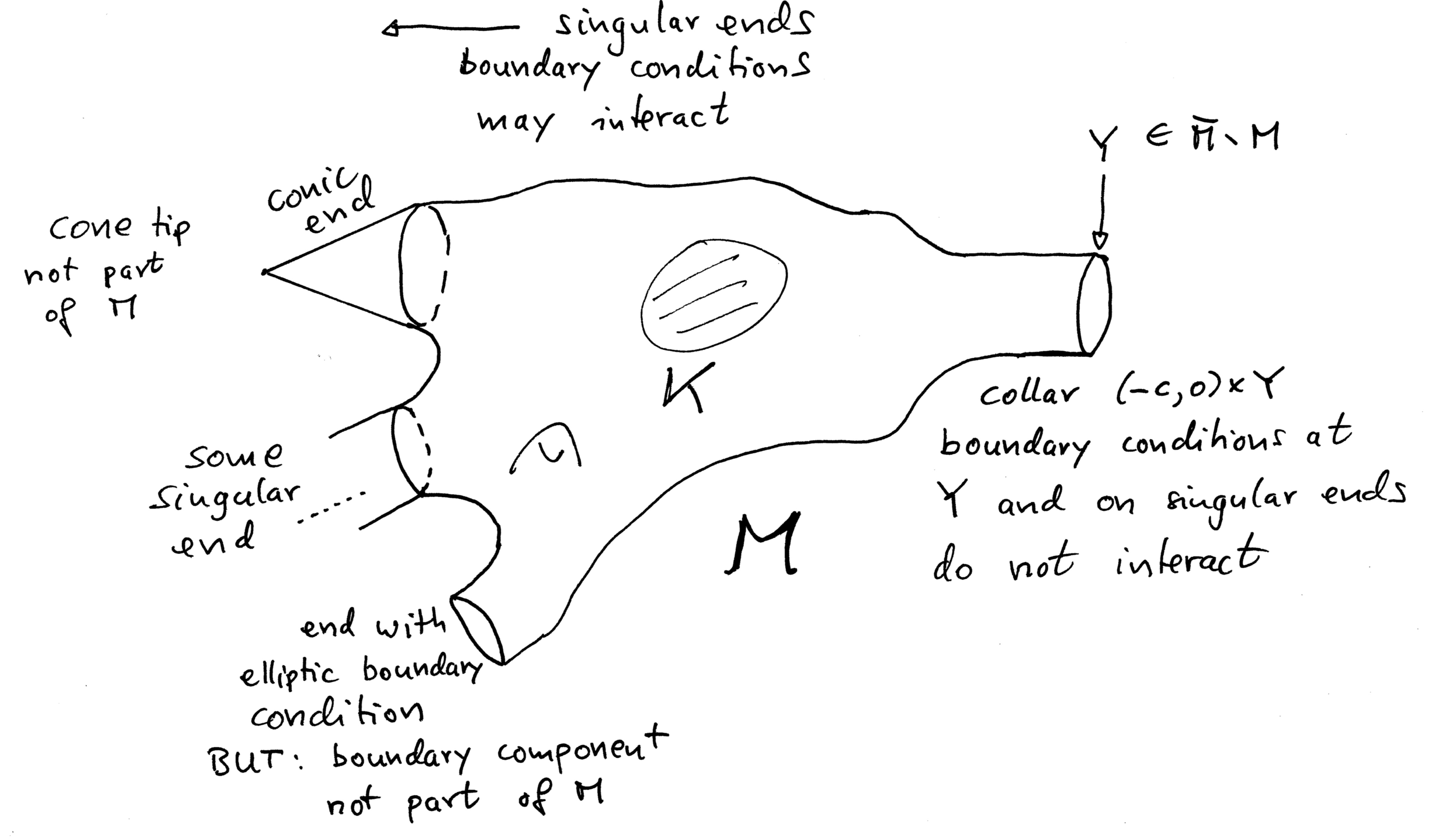}
\caption{\label{fig:singmani} Example of a singular manifold.}
\end{figure}
\subsection{Standing assumptions}\label{ss:EOGSA} Let $M^m$ be a Riemannian
manifold of dimension $m$; it is essential to note that $M^m$ is \emph{not}
necessarily complete, \cf Figure \ref{fig:singmani}. 
Furthermore, let $P_0:\ginfz{M,E}\longrightarrow \ginfz{M,E}$ be a second order
formally selfadjoint elliptic differential operator acting on the compactly
supported sections, $\ginfz{M,E}$,
of the Hermitian vector bundle $E$. We assume that $P_0$ is bounded below
and we fix once and for all a bounded below selfadjoint extension $P$
of $P_0$ in the Hilbert space of square-integrable sections $L^2(M,E)$, \eg the Friedrichs extension.

Later on we will need a class of operators which is slightly more general
than (pseudo)differential operators. For our purposes it will suffice
to consider an auxiliary operator $V$ which for each real $s$ maps
\begin{equation}\label{eq:EOG1}
V: H^s_{\loc}(M,E)\longrightarrow H^{s-1}_{\comp}(M,E)
\end{equation}
the space $H^s_{\loc}(M,E)$ of sections, which are locally of Sobolev class
$s$, continuously into the space of compactly supported sections of
Sobolev class $s-1$, \cf \cite[Sec.~I.7]{Shu:POS}.
We assume that $V$ is symmetric with respect to the $L^2$-scalar product
on $E$, \ie $\scalar{Vf}{g}=\scalar{f}{Vg}$ for $f\in H^1_{\loc}(M,E),
g\in L^2_{\loc}(M,E)$.

Finally, we assume that $V$ is confined to a compact subset $\cK\subset M$
in the sense that
\begin{equation}\label{eq:EOG2}
   M_\varphi V = V M_\varphi =0
\end{equation}
for any smooth function vanishing in a neighborhood of $\cK$. \Eqref{eq:EOG2}
implies that $V$ commutes with $M_\varphi$
for any smooth function which is \emph{constant} in a neighborhood of $\cK$.
Our main example is the operator $\tilde\Delta^\TT$ defined
after \Eqref{eq:GTP6} below.

In view of \Eqref{eq:EOG1} and the ellipticity of $P_0$, the operator
$V$ is $P$-bounded with arbitrarily small bound, thus $P+V$ is selfadjoint
and bounded below as well.

With regard to the mapping property \Eqref{eq:EOG1} of $V$
we introduce the space $\Opc^a(M,E)$ of linear operators $A$
mapping $H^s_{\loc}$ continuously into $H^{s-a}_{\comp}$
and whose Schwartz kernel $K_A$ is compactly supported.
Obvious examples are pseudodifferential operators with
compactly supported Schwartz kernel, but also certain
Fourier integral operators. The point is that elements
in $\Opc$ are not necessarily pseudolocal. Note that
$V$ is in $\Opc^1(M,E)$.

The set-up outlined in this Subsection \ref{ss:EOGSA} will be in
effect during the remainder of this Section \ref{s:EOG}. 

\subsection{Heat kernel estimates for $P+V$}\label{ss:HKE}

\begin{lemma}\label{l:EOG1}
For all $s\ge 0$ we have $\dom(P+V)^s=\dom(P^s)$.
Furthermore, the operator $e^{-t(P+V)}$, $t>0$,
has a smooth integral kernel.
\end{lemma}
\begin{proof}
 By complex interpolation \cite[Sec.~4.2]{Tay:PDEI} it suffices
 to prove the first claim for $s=k\in\N$ where it follows easily
 by induction exploiting the elliptic regularity for $P$ and
 \Eqref{eq:EOG1}.

 Consequently, $e^{-t(P+V)}$ is a selfadjoint operator which
 maps $L^2(M,E)$ into
 \begin{equation}
  \bigcap_{k\ge 0} \dom((P+V)^k) = \bigcap_{k\ge 0} \dom(P^k)
\end{equation}
and the latter is contained in $\ginf{M,E}$ by elliptic regularity.
This implies smoothness of the kernel of $e^{-t(P+V)}$.
\end{proof}
   
\begin{prop}\label{p:EOG2} Let $A\in \Opc^a(M,E), B\in \Opc^b(M,E)$
 with compactly supported Schwartz kernels $K_A, K_B$. 
 Denote by $\pi_j:M\times M\to M, j=1,2$, the projections onto
 the first resp.~second factor and suppose that
 $\pi_2(\supp K_A)\cap \pi_1(\supp K_B)=\emptyset$
 and $\pi_2(\supp K_A)\cap \cK=\emptyset$ (for $\cK$ \cf Subsec. \ref{ss:EOGSA}).

 Then $A e^{-t(P+V)} B$ is a trace class operator and 
 \begin{equation}
  \bigl\| Ae^{-t(P+V)} B \bigr\|_{\tr} = O(t^\infty), \quad t\to 0+.
 \end{equation}
\end{prop}
Here $O(t^\infty)$ is an abbreviation for $O(t^N)$ for any $N$;
the $O$-constant may depend on $N$. Furthermore, $\|\cdot\|_{\tr}$
denotes the trace norm on the Schatten ideal of trace class
operators.
\begin{proof} (\textit{cf.~}\cite[Sec.~I.4]{Les:OFT}).
Since the Schwartz kernels are compactly supported it suffices
 to prove that for all real $\ga,\gb$ and all $N>0$ we have
 \begin{equation}\label{eq:EOG3}
  \bigl\| Ae^{-t(P+V)} B \bigr\|_{\ga,\gb} = O(t^N), \quad t\to 0+.
 \end{equation}
 Here, $\|\cdot\|_{\ga,\gb}$ stands for the mapping norm between
 the Sobolev spaces $H^\ga(\pi_2(\supp K_B),E)$ and 
 $H^\gb(\pi_1(\supp K_A),E)$. The $O$-constant may depend on $A,B, \ga,\gb, N$.

\Eqref{eq:EOG3} follows from Duhamel's formula by a standard bootstrapping
 argument as follows: note first, that the mapping properties of $A,B$
 and $P+V$ imply that for real $\ga$
 \begin{equation}\label{eq:EOG4}
  \bigl\| Ae^{-t(P+V)} B \bigr\|_{\ga,\ga-a-b} = O(1), \quad t\to 0+.
 \end{equation}

Assume by induction that for fixed $l, N$,
for all $A,B$ satisfying our assumptions and for all
real $\ga$
 \begin{equation}\label{eq:EOG5}
  \bigl\| Ae^{-t(P+V)} B \bigr\|_{\ga,\ga-a-b+l} = O(t^N), \quad t\to 0+.
 \end{equation}
 Fix plateau functions $\chi,\varphi,\psi\in \cinfz{M}$ with the following
 properties:
 \begin{enumerate}
  \item $\varphi \equiv 1$ in a neighborhood  of $\pi_2(\supp K_A)$ and
   $\supp \varphi \cap \cK =\emptyset$.
  \item $\psi\equiv 1$ in a neighborhood of $\pi_1(\supp K_B)$. 
  \item $\chi\equiv 1$ in a neighborhood of $\supp \varphi$ and $\supp
   \chi\cap \cK=\emptyset$.
  \item $\supp \chi\cap \supp \psi=\emptyset$.
 \end{enumerate} 
Then
\begin{multline}
 \bigl\| Ae^{-t(P+V)} B \bigr\|_{\ga,\ga-a-b+l+1/2} = 
       \bigl\| A \varphi e^{-t(P+V)} \psi  B \bigr\|_{\ga,\ga-a-b+l+1/2} \\
       \le C_1  \bigl\| \varphi e^{-t(P+V)} \psi  \bigr\|_{\ga-b,\ga-b+l+1/2}.
\end{multline}
From
\begin{equation}
 \bigl(\pl_t + P+V \bigr) \varphi e^{-t(P+V)} \psi
      =\chi [P_0,\varphi] e^{-t(P+V)} \psi,
\end{equation}
where $[P_0,\varphi]$ denotes the commutator between the differential
expression $P_0$ and multiplication by $\varphi$, we infer
\begin{equation}
\varphi e^{-t(P+V)} \psi 
   = \int_0^t \chi e^{-(t-s)(P+V)} \chi [P_0,\varphi] e^{-s(P+V)} \psi ds;
\end{equation}
here we have used the assumptions on the support of $\chi,\psi,\varphi$ and
\Eqref{eq:EOG2}. In the displayed formulas we wrote, to save some space,
$\chi,\psi,\varphi$ for the multiplication operators
$M_\chi,M_\psi,M_\varphi$, resp.

For $\tilde\ga=\ga-b$ we now find
\begin{multline}
 \bigl\| \varphi e^{-t(P+V)} \psi \bigr\|_{\tilde\ga,\tilde \ga+l+1/2}\\
    \le \int_0^t \bigl\| \chi e^{-(t-s)(P+V)} \chi
    \bigr\|_{\tilde\ga-1+l,\tilde\ga+l+1/2} \bigl\|  [P_0,\varphi] e^{-s(P+V)}
    \psi \bigr \|_{\tilde\ga,\tilde\ga-1+l} ds.
\end{multline}
Since $[P_0,\varphi]$ is in $\Opc^1$ we find using \Eqref{eq:EOG5}
\begin{equation}\label{eq:1209103}  
    \bigl\|  [P_0,\varphi] e^{-s(P+V)} \psi \bigr \|_{\tilde\ga,\tilde\ga-1+l}
    = O(s^N), \quad \text{as } s\to 0+.
\end{equation}
Furthermore, denoting by $C$ a constant such that $P\ge -C+1$,
\begin{align}   
    \bigl\| &\chi e^{-u(P+V)} \chi
    \bigr\|_{\tilde\ga-1+l,\tilde\ga+l+1/2}\nonumber \\
     \le& \bigl\| (P+V+C)^{(\tilde\ga+l-1)/2} \chi\bigr\|_{\tilde \ga
    -1+l,0}\,\cdot\, \label{eq:1209105}\\
    &\quad \cdot\,\bigl\|(P+V+C)^{3/4} e^{-u(P+V)} \bigr\|_{0,0}\, \cdot\,
    \bigl\|\chi (P+V+C)^{-(\tilde\ga+l+1/2)/2}\bigr\|_{0,\tilde \ga+l+1/2}.
    \nonumber
\end{align}
The first and the third factor on the right are bounded while for the second
factor we have by the Spectral Theorem
\begin{equation}\label{eq:1209106}  
 \bigl\|(P+V+C)^{3/4} e^{-u(P+V)} \bigr\|_{0,0}=O(u^{-3/4}),\quad \text{as }
 u\to 0+.
\end{equation}
Thus
\begin{equation}
 \bigl\| \varphi e^{-t(P+V)} \psi \bigr\|_{\tilde\ga,\tilde \ga+l+1/2}
  \le C_1 \int_0^t (t-s)^{-3/4} s^N ds = O( t^{N+1/4} ), \quad t\to 0+.
\end{equation}
Thus we have improved the parameters $l$ and $N$ in \Eqref{eq:EOG5}
by $1/2$ resp. $1/4$ and therefore the result follows by induction.
\end{proof} 

\begin{prop}\label{p:EOG3} Under the Standing Assumptions \ref{ss:EOGSA}
 let $\varphi,\psi\in \cinf{M}$ with
 $\supp \varphi\cap \supp\psi$ being compact (the individual supports of
 $\varphi$ or $\psi$ may be non-compact!)
 such that $d\varphi, d\psi$ are compactly supported and that
 $\supp d\varphi \cap \cK=\emptyset= \supp d\psi\cap \cK$. Furthermore,
 assume that multiplication by $\varphi$ and by $\psi$ preserves
 $\dom(P+V)=\dom(P)$. 

 Then for $t>0$ the operator $\varphi e^{-t(P+V)} \psi$ is trace class
 and 
 \begin{equation}
  \bigl \| \varphi e^{-t(P+V)} \psi\bigr\|_{\tr} =O ( t^{-m/2-0}), \quad t\to
  0+.
 \end{equation}  
If $\supp\varphi\cap \supp \psi=\emptyset$ then the right hand side
can be improved to $O(t^\infty), t\to 0+$.
\end{prop} 
Here $O(t^{-m/2-0})$ is an abbreviation for $O(t^{-m/2-\eps})$ for any $\eps>0$;
the $O$-constant may depend on $\eps$. 
\begin{proof} Assume first that additionally $\psi$ is compactly supported.
Again applying Duhamel we find
\begin{equation}\label{eq:EOG6}
\varphi e^{-t(P+V)} \psi 
   = \int_0^t  e^{-(t-s)(P+V)} [P_0,\varphi] e^{-s(P+V)} \psi ds.
\end{equation}
Now apply Lemma \ref{l:EOG1} and Proposition \ref{p:EOG2} to 
the operator $[P_0,\varphi] e^{-s(P+V)} \psi$. If $\supp\varphi\cap
\supp\psi\not=\emptyset$ then the trace norm estimate is a simple
consequence of Sobolev embedding and the established mapping properties.
If $\supp \varphi\cap\supp\psi=\emptyset$ then Proposition \ref{p:EOG2}
implies $\bigl\| [P_0,\varphi]  e^{-s(P+V)} \psi \bigr\|_{\tr}=O(t^\infty)$
and the claim follows in this case.

Since $e^{-t(P+V)}$ is selfadjoint the roles of $\varphi,\psi$ may be
interchanged by taking adjoints and hence the Proposition is proved
if $\varphi$ \emph{or} $\psi$ is compactly supported. The general
case now follows from formula \Eqref{eq:EOG6} since the compactness
of $\supp d\varphi$ implies the compactness of the support of the
Schwartz kernel of $[P_0,\varphi]$.
\end{proof}

\subsection{Operator Gluing}\label{ss:OG}
\begin{figure}
\includegraphics[width=10cm]{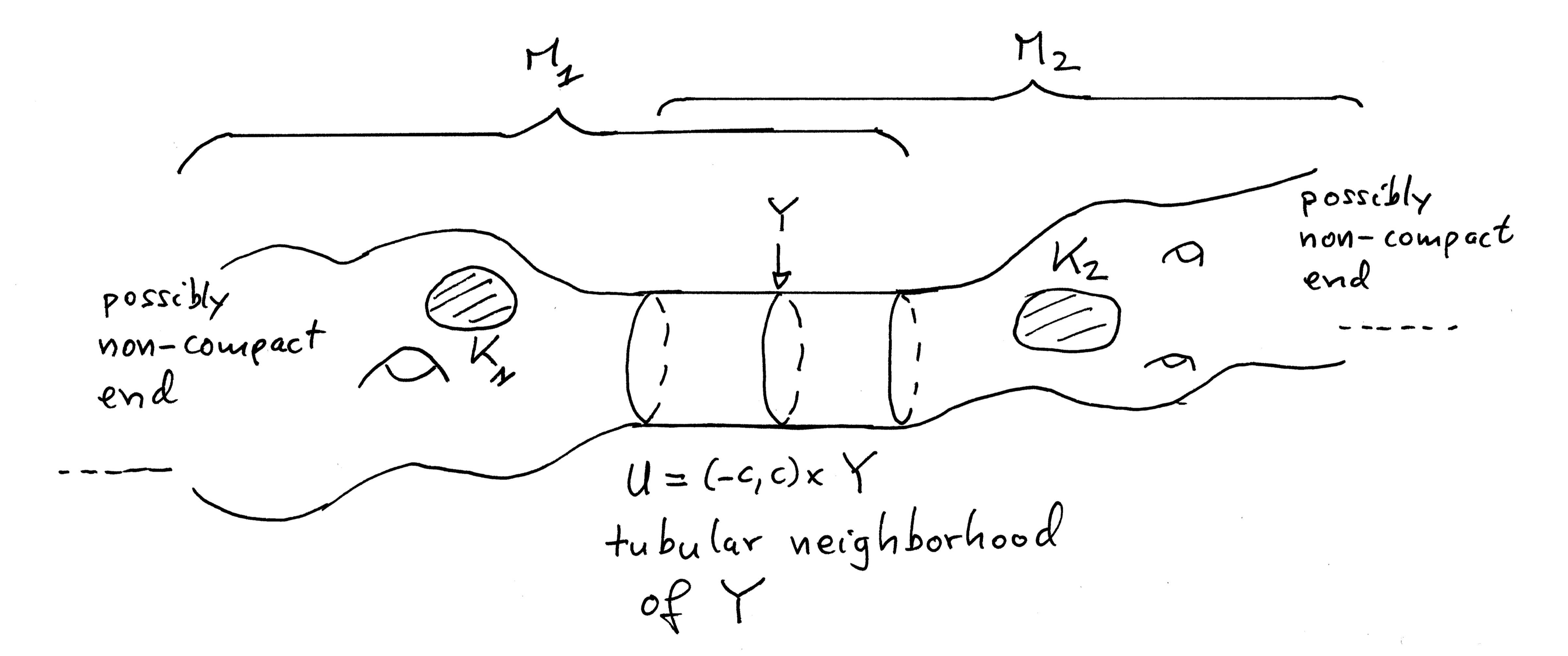}
\caption{\label{fig:gluing} The gluing situation.}
\end{figure}
Now we assume that we have two triples $(M_j,P_j^0,V_j), j=1,2$
consisting of Riemannian manifolds $M_j^m$ and operators $P_j^0$, $V_j$
satisfying the Standing Assumptions \plref{ss:EOGSA}.

Furthermore, we assume that each $M_j$ is the interior of a manifold
$\ovl{M_j}$ with compact boundary $Y$ (it is essential that $\ovl{M_j}$
is not necessarily compact). Let $U=Y\times (-c,c)$ be a common collar
of $Y$ in $M_1$ resp. $M_2$ such that $\pl M_1=Y\times \{1\}$ and 
$\pl M_2=Y\times \{-1\}$.

We assume that the sets $\cK_j$ corresponding to $V_j$ (\textit{cf.~}\Eqref{eq:EOG2})
lie in $M_j\setminus U$ and that $P_1^0$ coincides with $P_2^0$ over $U$.
Then $P_j^0$ and $V_j$ give rise naturally to a differential operator
$P^0=P_1^0\cup P_2^0$ on $M:=\bigl(M_1\setminus (Y\times (0,c))\bigr)
\cup_{Y\times \{0\}} \bigl(M_2\setminus (Y\times (-c,0))\bigr)$
resp. $V=V_1+V_2\in \Opc^1(M,E)$, where $E$ is the bundle obtained
by gluing the bundles $E_1$ and $E_2$ in the obvious way. Note that
due to \Eqref{eq:EOG2} the operators $V_1, V_2$ extend to $M$ in
a natural way.

\begin{figure}
 \includegraphics[width=10cm]{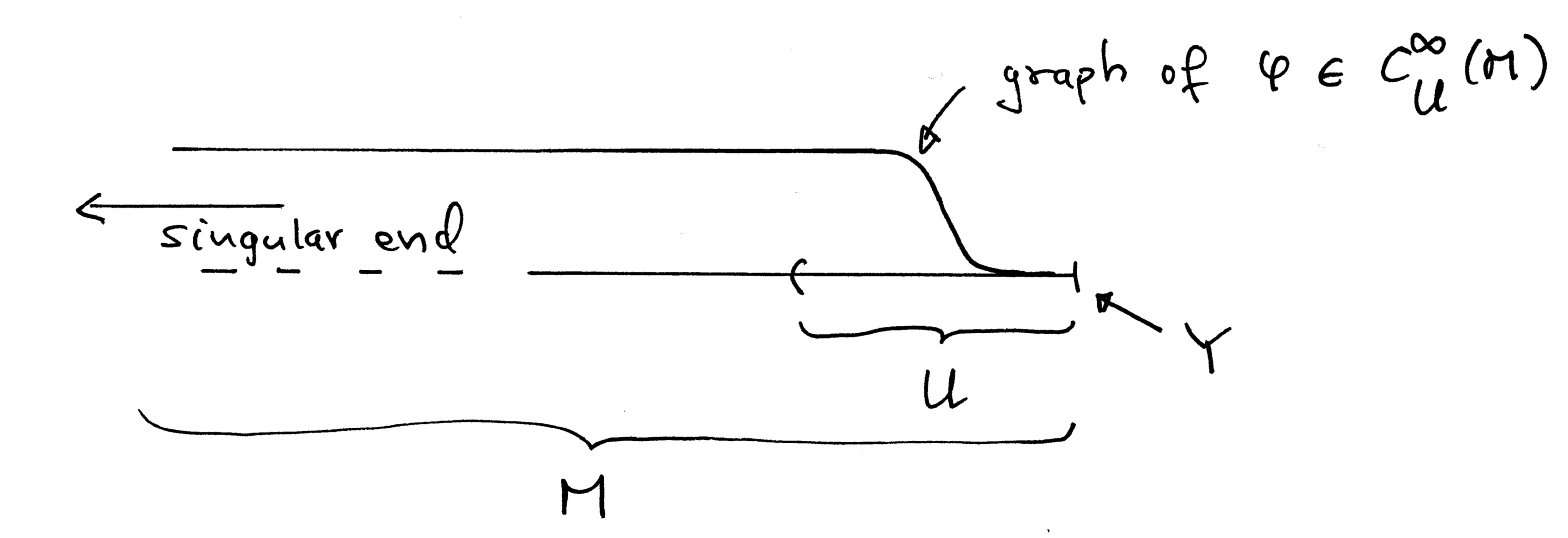} 
 \caption{\label{fig:cinfU}
 Schematic sketch of a function in $\CU{M}$. The line indicates
 the manifold $M$, to the left there are the possible non-complete
 ends. On the right there is the collar $U$.
}
\end{figure} 
\begin{dfn}\label{def:EOG4} By $\CU{M_j}$ we denote the space
 of those smooth functions $\varphi\in\cinf{M_j}$ such that
 $\varphi$ is constant in a neighborhood of $M_j\setminus U$
 and $\varphi\equiv 0$ in a neighborhood of $\pl \ovl{M_j}$,
 \cf Figure \ref{fig:cinfU}.
\end{dfn}
A function $\varphi\in\CU{M_j}$ extends by $0$ to a smooth
function on $M$.

\begin{prop}\label{p:EOG5} Let $P_j, j=1,2$, be closed symmetric
 extensions of $P_j^0$ which are bounded below and for which
 \begin{equation}\label{eq:EOG7}
  \varphi \dom(P_j^*) \subset\dom(P_j), \quad\text{for all } \varphi\in
  \CU{M_j}.
 \end{equation}
 Put for a fixed pair of functions $\varphi_j\in \CU{M_j}, j=1,2$
 \begin{equation}\label{eq:EOG8}
  \begin{split}
 \dom(P):=&\bigsetdef{ f\in \dom(P_{\max}^0)}{ \varphi_j f \in \dom(P_j), \;
 j=1,2 }\\
         =&H^2_{\comp}(U,E) +\varphi_1 \dom(P_1)+ \varphi_2 \dom(P_2).
 \end{split}
 \end{equation}
$\dom(P)$ is indeed independent of the particular choice of $\varphi_j$
and the operator $P$ which is defined by restricting $P_{\max}^0=(P^0)^*$ to
$\dom(P)$ is selfadjoint and bounded below. $V$ is $P$--bounded with
arbitrarily small bound and hence $P+V$ is selfadjoint and bounded below
as well.

Furthermore, if for fixed $j\in\{1,2\}$ we have $\varphi, \psi\in \CU{M_j}$
satisfying \Eqref{eq:EOG7} then $\varphi e^{-t(P_j+V)}\psi - \varphi
e^{-t(P+V)}\psi$ is trace class and its trace norm is $O(t^\infty)$ as $t\to
0+$.
\end{prop}

\begin{remark} 1. Note that it is not assumed that $\varphi e^{-t(P_j+V)}\psi$ or $\varphi
e^{-t(P+V)}\psi$ is of trace class individually!

2. \Eqref{eq:EOG7} says that the ``boundary conditions" at the exits of $M_1$
and $M_2$ are separated. Let us illustrate this by an example:
let $M_1=(-1,1/2), M_2=(-1/2,1), U=(-1/2,1/2), M=(-1,1)$ and
$P_j^0=-\frac{d^2}{dx^2}=\Delta, j=1,2,$ the Laplacian on functions.
Let $P_1^{\textup{per}}$ be the Laplacian $\Delta$ on $M_1$ with
periodic boundary conditions. These boundary conditions are not separated
and indeed for $\varphi\in \cinf{-1,1/2}$ with $\varphi(x)=1$ for $x\le -1/4$
and $\varphi(x)=0$ for $x\ge 1/4$ the space $\varphi\dom(P_1^{\textup{per}})$
equals $\varphi H^2[-1,1/2]$ and this is not contained in
$\dom(P_1^{\textup{per}})$.

However, for any pair of selfadjoint extensions $P_j$ of $P_j^0, j=1,2$ with separated
boundary conditions at the ends of the intervals $M_j$ one has
$\varphi\dom(P_j)\subset\dom(P_j)$, \ie the condition \Eqref{eq:EOG7} is satisfied
and Proposition \plref{p:EOG5} applies to this pair.
\end{remark}

\begin{proof} Since $H^2_{\comp}(U,E)\subset \dom(P_{j,\min}^0)$ the second
 equality in \Eqref{eq:EOG8}, the symmetry of $P$ and the independence of
 $\dom(P)$ of the particular choice of $\varphi_j$ are easy consequences
 of \Eqref{eq:EOG7}.

 To prove selfadjointness let $f\in \dom(P^*)$. We claim that for
 $\varphi_1\in\CU{M_1}$ we have $\varphi_1 f\in \dom(P_1^*)$. Indeed for
 $g\in\dom(P_1)$ we have
 \begin{equation}
    \scalar{\varphi_1 f}{P_1 g} = \scalar{f}{\ovl{\varphi_1} P_1 g}
     = \scalar{f}{[\ovl{\varphi_1},P_1^0]g} + \scalar{f}{P \ovl{\varphi_1}
    g}.
 \end{equation}
Since $\supp d\varphi_1\subset U$ is compact and since
$[\ovl{\varphi_1},P_1^0]$ is a compactly supported first order differential
operator on $U$ we find
\begin{equation}
   \ldots = \scalar{[P_1^0,\varphi_1] f+ \varphi_1 P^* f}{g}
\end{equation}
proving $\varphi_1 f\in\dom(P_1^*)$. In view of \Eqref{eq:EOG7} we see, by
choosing another plateau function $\psi\in\CU{M_1}$ with
$\psi\varphi_1=\varphi_1$ that $\varphi_1 f\in \dom(P_1)$.
In the same way we conclude $\varphi_2 f\in\dom(P_2)$ for
$\varphi_2\in\CU{M_2}$ and thus $f\in\dom(P)$. 

To prove the trace class property and the trace estimate we choose
another plateau function $\chi\in\CU{M_j}$ such that $\chi\equiv 1$
in a neighborhood of $\supp\psi$ with $\chi-\psi\in \cinfz{M_j}$;
hence $\chi$ also satisfies \Eqref{eq:EOG7}.

Consider first $K_t:= \chi e^{-t(P_j+V_j)} \psi - \chi e^{-t(P+V)}\psi$.
$K_{t=0}=0$ and
\begin{equation}\label{eq:EOG11}   
  \begin{split}
   \bigl(\pl_t + P +V\bigr) K_t = [P_j^0,\chi] e^{-t(P_j+V_j)}\psi -
   [P^0,\chi] e^{-t(P+V)}\psi.
  \end{split}
\end{equation}
Here we have used that multiplication by $\chi$ commutes with $V, V_j$,\cf
\Eqref{eq:EOG2}. Propositions \plref{p:EOG2} and \plref{p:EOG3} now imply
that $K_t$ is trace class for $t>0$ and that $\|K_t\|_{\tr}=O(t^\infty)$
as $t\to 0+$.
Consequently 
\[
   \| \chi \varphi e^{-t(P_j+V)}\psi - \chi \varphi e^{-t(P+V)}\psi\|_{\tr}
   \le \|\varphi\|_\infty \, \|K_t\|_{\tr}=O(t^\infty).
\]
To $(1-\chi) \varphi e^{-t(P_j+V)}\psi - (1-\chi) \varphi e^{-t(P+V)}\psi$
we can apply Proposition \plref{p:EOG3} since $(\supp \psi) \cap \supp
(1-\chi)=\emptyset$ and the proof is complete.
\end{proof}

Finally, we discuss heat expansions. Under the assumptions of Proposition
\ref{p:EOG5} assume that $P_j+V_j$ has \emph{discrete dimension spectrum outside
$U$}. By this we understand that for $\varphi\in\CU{M_j}$ the operator
$\varphi e^{-t(P_j+V_j)}$ is trace class and that there is an asymptotic
expansion of the form \Eqref{eq:HeatExpansion} with $a_{\ga k}=a_{\ga
k}(\varphi)$. Then

\begin{cor}\label{cor:EOG6} Under the additional assumption of discrete
 dimension spectrum for $P_j+V_j$ outside $U$ the operator $P+V$
 has discrete dimension spectrum and for any $\varphi\in\CU{M_1}$ we
 have
 \begin{equation}\label{eq:EOG20}
  \Tr\bigl(e^{-t(P+V)}\bigr)=\Tr\bigl(\varphi e^{-t(P_1+V_1)}\bigr)+
  \Tr\bigl((1-\varphi)e^{-t(P_2+V_2)}\bigr)+O(t^\infty)
\end{equation} 
as $t\to 0+$.
\end{cor}

\begin{proof} \Eqref{eq:EOG20} is immediate from Proposition \ref{p:EOG5} and
 the discrete dimension spectrum assumption.\end{proof}
 
 We add, however, a little
 more explanation since the term ``discrete dimension spectrum outside $U$''
 might lead to some confusion: since $\cK\cap U=\emptyset$ (\textit{cf.}
 \Eqref{eq:EOG2} and the second paragraph of this Subsection \ref{ss:OG})
 for $f\in\ginfz{U,E}$ we have $(P+V)f=Pf$. The classical interior
 parametric elliptic calculus (\eg \cite{Shu:POS}) then implies that
 for $\varphi\in \cinfz{U}$ there is an asymptotic expansion
 \begin{equation}
  \Trbig{\varphi e^{-t(P+V)}}\sim_{t\searrow 0} \sum_{j\ge 0} a_j(P,\varphi)
  \, t^{j-m/2},
 \end{equation} 
where $a_j(P,\varphi)=\int_M \tilde a_j(x,P) \varphi(x) dx$ and $\tilde
a_j(x,P)$ are the local heat invariants of $P$. 
Thus over any compact subset in the \emph{interior} of $M\setminus \cK$
the discrete dimension spectrum assumption follows from standard
elliptic theory and hence is a non-issue. Rather it is a condition
on the behavior of $P$ on non-compact ``ends'' and a condition on $V$
over $\cK$.

\subsection{Ideal boundary conditions with discrete dimension
spectrum}\label{ss:IBC}
The remarks of the previous Subsection extend to ideal boundary conditions
of elliptic complexes in a straightforward fashion. Let $X$ be a Riemannian
manifold which is the interior of a Riemannian manifold $\ovl{X}$ with
compact boundary Y, and let $U=(-c,0)\times Y$ be a collar of the boundary. 
Since $\ovl{X}$ is allowed to be non-compact it is not excluded that away from $U$ there
are ``ends'' of $\ovl{X}$ which can be completed by adding another boundary
component, see Figure \ref{fig:singmani}.

As an example which illustrates what can happen consider a compact manifold
$Z$ with boundary, where $\pl Z=Y_1\cup Y_2 \cup Y_3$ consists of the disjoint
union of three compact closed manifolds $Y_j, j=1,2,3$. Attach a cone
$C(Y_3)=Y_3\times (0,1)$ with metric $dr^2+r^2 g_{Y_3}$ to $Y_3$ (and smooth
it out near $Y_3\times \{1\}$). Then put $X:=\bigl(Z\setminus {(Y_1\cup
Y_2)}\bigr) \cup_{Y_3} C(Y_3)$ and $\ovl{X}:= \bigl(Z\setminus {Y_2}\bigr)\cup_{Y_3} C(Y_3)$.
Then $Y_1$ plays the role of $Y$ above, but $\overline{X}$ is not compact.
Cf. Figure \ref{fig:singmani}.

When introducing closed extensions (viz.~boundary conditions) for elliptic
operators on $X$ it is important that the boundary conditions at $Y_1$ and
$Y_2$ resp. the cone do \emph{not} interact in order to ensure \Eqref{eq:EOG7}
to hold.

Leaving this example behind let $(\Gamma_c^\infty(E),d)$ be an elliptic complex
and let $(\sD,D)$ be an ideal boundary condition for $(\Gamma_c^\infty(E),d)$. That is
a Hilbert complex such that $D_j$ are closed extensions of $d_j$.

We say that the ideal boundary condition
$(\sD,D)$ has discrete dimension spectrum outside $U$ if the Laplacians
$\Delta_j=D_j^*D_j+D_{j-1}D_{j-1}^*$ have discrete dimension spectrum outside $U$, \cf
the paragraph before Corollary \ref{cor:EOG6}.
Then Proposition \ref{p:EOG5} and Corollary \ref{cor:EOG6} hold for the
Laplacians.

More concretely, let $X, Y$ be as before and let $(F,\nabla)$ be a flat bundle
over $\ovl{X}$. Assume that we are given an ideal boundary condition $(\sD,D)$ of the de Rham complex
$(\Omega^\bullet(X;F),d)$ with values in the flat bundle $F$ with discrete dimension
spectrum over the open set $X\setminus U$, $U=(-c,0)\times Y$.
Fix a smooth function $\varphi\in\cinf{-c,0}$ which is $1$ near $-c$
and $0$ near $0$ and extend it to a smooth function on $X$ in the obvious way.

We then define the absolute and relative boundary conditions at $Y$ as follows.
\begin{equation}\label{eq:AbsRelBoundaryCondition}
    \begin{split}
        \sD^j(X;F)&:=\varphi \dom(D_j)+(1-\varphi)\dom(d_{j,\max}),\\
        \sD^j(X,Y;F)&:=\varphi \dom(D_j)+(1-\varphi)\dom(d_{j,\min}).
    \end{split}
\end{equation}
Since the Laplacians of the maximal and minimal ideal boundary condition
are near $Y$ realizations of local elliptic boundary conditions
(\textit{cf., e.g.},
\cite[Sec.~2.7]{Gil:ITH}) it follows from Prop. \ref{p:EOG5} and Corollary \ref{cor:EOG6} 
applied to $M_1=X, M_2=Y\times (-c,0), U=Y\times (-c,-c/2)$ 
that the Hilbert complexes $(\sD(X;F),d)$ and $(\sD(X,Y;F),d)$
are Hilbert complexes with discrete dimension spectrum.

\section{Vishik's moving boundary conditions}\label{s:VMB}    
\subsection{Standing Assumptions}\label{ss:VMBSA}
We discuss here Vishik's \cite{Vis:GRS} moving boundary conditions for the de Rham complex in
our slightly more general setting.
Let $X$ be a Riemannian manifold (not necessarily compact or complete!), see
Figure \ref{fig:singmani}. Furthermore, let $(F,\nabla)$
be a flat bundle with a (not necessarily flat) Hermitian metric $h^F$. We assume
furthermore, that $X$ contains a compact separating hypersurface $Y\subset X $
such that in a collar neighborhood $W=(-c,c)\times Y$ all structures are product.
In particular we assume that $\nabla^F$ is in temporal gauge on $W$, that is
$\nabla^F\restriction W=\pi^* \tilde \nabla^F$ for a flat connection
$\tilde\nabla^F$ on $F\restriction Y$,
$\pi$ denotes the natural projection map $W\to Y$. In other words $X$ is obtained
by gluing two manifolds with boundary $X^\pm$ along their common boundary $Y$ where
all structures are product near $Y$, \cf Figure \ref{fig:gluing}.

We make the fundamental assumption that
\begin{equation}\label{eq:VMBSA1}\parbox{\quoteboxwidth}{
we are given ideal boundary conditions $(\sD^\pm,D^\pm)$ of the twisted de Rham
complexes $(\Omega^\bullet(X^{\circ,\pm};F),d)$ which have discrete dimension
spectrum over $U^\pm:=X^\pm\setminus W$. We put $\Xcut:=X^-\coprod X^+$.}
\end{equation}

\subsection{Some exact sequences and the main deformation result}\label{ss:Main}

As explained in Subsection \ref{ss:IBC} we therefore have the following Hilbert complexes with
discrete dimension spectrum: $\sD^\bullet(X^\pm;F)$ (absolute boundary condition at $Y$), 
$\sD^\bullet(X^\pm,Y;F)$ (relative boundary condition at $Y$), $\sD^\bullet(X;F)$ (continuous transmission
condition at $Y$).
By construction we have the following exact sequences of Hilbert complexes
\begin{equation}\label{eq:ExactSeq1}
\xymatrix{
    0 \ar[r]   &\leftrelcomplex \ar@{^{(}->}[r]^-{\ga_-}  &\totalcomplex \ar[r]^\beta &\rightabscomplex \ar[r] &  0,
}
\end{equation}
\begin{equation}\label{eq:ExactSeq2}
 \xymatrix{
    0 \ar[r]   &\genrelcomplex \ar@{^{(}->}[r]^-{\gamma_\pm}  &\genabscomplex \ar[r]^{i_\pm^*} &\bdycomplex \ar[r] &  0,
}
\end{equation}

\begin{equation}\label{eq:ExactSeq3}
 \xymatrix{
    0 \ar[r]   &\leftrelcomplex \oplus \rightrelcomplex \ar@{^{(}->}[r]^-{\ga_++\ga_-} &\totalcomplex \ar[r]^{r} &\bdycomplex \ar[r] &  0.
}
\end{equation}
Here $\ga_\pm$ are extension by $0$, $\gb$ is pullback (i.e. restriction) to $X^+$,
$\gamma_\pm$ is the natural inclusion of the complex $\genrelcomplex$ with relative
boundary condition at $Y$ into the complex $\genabscomplex$ with absolute boundary
condition, and $i_\pm:Y\hookrightarrow X^\pm$ is the inclusion map. Finally
$r\go=\frac{\sqrt{2}}{2}(i_+^* \go + i_-^*\go)=\sqrt{2} i_\pm^*\go$ for $\go\in\sD^\bul(X;F)$.

It is a consequence of standard Trace Theorems for Sobolev spaces that
$i_\pm^*: \genabscomplex \to \bdycomplex$ is well--defined,
see \eg \cite[Sec.~1]{Paq:PMP}, \cite{LioMag:NHB}, \cite{BruLes:BVPI}.
To save some space we have omitted the operator $D$ from the notation in all
the complexes in \Eqref{eq:ExactSeq1}--\Eqref{eq:ExactSeq3}. Clearly, the
complex differential is always the exterior derivative on the indicated
domains.

Each of the complexes \eqref{eq:ExactSeq1}, \eqref{eq:ExactSeq2}, \eqref{eq:ExactSeq3}
induces a long exact sequence in cohomology. We abbreviate these long exact
cohomology sequences by $\longexactcohone$, $\longexactcohtwo$, 
$\longexactcohthree$, resp. 
The long exact cohomology sequences of the complexes 
\eqref{eq:ExactSeq1}, \eqref{eq:ExactSeq2}, \eqref{eq:ExactSeq3} are exact sequences
of finite-dimensional \emph{Hilbert spaces} and therefore their torsion
$\tau(\sH(\ldots))$ is defined, \cf \Eqref{eq:OMZ14}. 
The Euler characteristics, \cf \Eqref{eq:EulerCharacteristic}, of the
complexes in \Eqref{eq:ExactSeq1}--\eqref{eq:ExactSeq3} are denoted
by $\chi(X^\pm,Y;F), \chi(X^\pm;F), \chi(X;F), \chi(Y;F)$ etc.

Next we introduce parametrized versions of the exact sequences
\Eqref{eq:ExactSeq1} and \eqref{eq:ExactSeq3}. The idea
is due to Vishik \cite{Vis:GRS} who applied it to give a new proof of the Ray--Singer
conjecture for compact smooth manifolds with boundary.
Namely, for 
$\theta\in\R$ consider the following ideal boundary condition of the twisted
de Rham complex on the disjoint union $\Xcut=X^-\coprod X^+$:
\begin{multline}\label{eq:1007096}
    \sD_\TT^j(X;F)\\
    :=\bigl\{(\go_1,\go_2)\in \sD^j(X^-;F)\oplus \sD^j(X^+;F)\bigm|
	    \cost\cdot i_-^*\go_1=\sint\cdot i_+^* \go_2\bigr\}.    
\end{multline}
      
We will see that for each real $\TT$ the complex $(\sD_\TT^\bul(X;F),d)$
is indeed a Hilbert complex with discrete dimension spectrum. In fact near
$Y$ it is a realization of a local elliptic boundary value problem for de Rham
complex on the manifold $\Xcut$; and away from $Y$ we may apply Corollary
\ref{cor:EOG6} and our assumption \Eqref{eq:VMBSA1} that the Hilbert complexes
$(\sD^\pm,D^\pm)$ have discrete dimension spectrum over $X^\pm\setminus W$.

Furthermore, for $\TT=0$ we have $\sD_\TT^\bul(X;F)=\sD^\bul(X^-,Y;F)\oplus \sD^\bul(X^+;F)$,
and for $\TT=\pi/4$ we see that (cf. \cite[Prop.~1.1 p.~16]{Vis:GRS}) the total Gau{\ss}--Bonnet
operators $d+d^*$ of the complexes $\sD_{\pi/4}(X;F)$ and $\sD(X;F)$ coincide.
Hence the family of complexes $(\sD_\TT^\bul(X;F), d^\theta)$ interpolates in a sense
between the direct sum $\sD^\bul(X^-,Y;F)\oplus \sD^\bul(X^+;F)$ and
the complex $\sD^\bullet(X;F)$ on the manifold $X$. 

The parametrized versions of \eqref{eq:ExactSeq1}, \eqref{eq:ExactSeq3} are
then
\begin{equation}\label{eq:ExactSeq1p}
\xymatrix{
    0 \ar[r]   &\leftrelcomplex \ar@{^{(}->}[r]^-{\ga_\TT}  &\thetacomplex \ar[r]^{\beta_\TT}
     &\rightabscomplex \ar[r] &  0,
}
\end{equation}
\begin{equation}\label{eq:ExactSeq3p}
 \xymatrix{
    0 \ar[r]   &\leftrelcomplex \oplus \rightrelcomplex \ar@{^{(}->}[r]^-{\gamma_++\gamma_-} 
      &\thetacomplex \ar[r]^{r_\TT} &\bdycomplex \ar[r] &  0,
}
\end{equation}
where $\ga_\TT\go=(\go,0)$ is extension by $0$, $\gb_\TT(\go_1,\go_2)=\go_2$ is 
restriction to $X^+$, $\gamma_+\oplus \gamma_-(\go_1,\go_2)=(\go_1,\go_2)$ is
inclusion and $r_\TT(\go_1,\go_2)=\sint\cdot i_-^*\go_1+\cost\cdot i_+^*\go_2.$
Let $\longexactcohonetheta, \longexactcohthreetheta$ be the corresponding long exact
cohomology sequences. 

We denote the cohomology groups of the complex $\thetacomplex$
by $H_\TT^j(X;F)$; the corresponding space of harmonic forms
will be denoted by $\hat H_\TT^j(X;F)$. 
For the next result we need some more notation. Let $\sH$ be a Hilbert space
and let $T:\sH\to\sH$ be a bounded linear operator. For a finite-dimensional
subspace $V\subset \sH$ we write $\Tr(T\restriction V)$ for $\Tr(P_V T P_V)$
where $P_V$ is the orthogonal projection onto $V$. If $e_j$, $j,\ldots,n$,
is an orthonormal basis of $V$ then
\begin{equation}\label{eq:1007097}
     \Tr(T\restriction V)=\sum_{j=1}^n \scalar{T e_j}{e_j}.    
\end{equation}
We will apply this to $\gb_\TT$ on the space $H_\TT^j(X;F)$. If $e_j$,
$j,\ldots,n$,
is an orthonormal basis of $\hat H^j_\TT(X;F)$
then
\begin{equation}\label{eq:1007098}
    \trbetatheta = \sum_{j=1}^n \| e_j\restriction X^+\|_{X^+}^2
                 = \sum_{j=1}^n \int_{X^+} e_j\wedge * e_j.
\end{equation}

After these preparations we are able to state our main technical result. It is
inspired by Lemma 2.2 and Section 2.6 in \cite{Vis:GRS}.
\begin{theorem}\label{thm:Theorem2}
The functions $\TT\mapsto \log T(\thetacomplex)$,
$\log \tau(\longexactcohonetheta)$, $\log \tau(\longexactcohthreetheta)$
are differentiable for $0<\TT<\pi/2$ . Moreover, for $0<\TT<\pi/2$
\begin{align}
\frac{d}{d\TT}&\log T(\thetacomplex) =\nonumber\\
       & =\frac{2}{\sin 2\TT}\Bigl[- \altsum \trbetatheta + \chi(X^+;F)\Bigr]-
       \tan\TT\cdot \chi(Y;F),\label{eq:Theorem2Eq1}\\
\frac{d}{d\TT}&\log \tau\bigl(\longexactcohonetheta\bigr) =\nonumber\\
      &  =\frac{2}{\sin 2\TT}\Bigl[- \altsum \trbetatheta + \chi(X^+;F)\Bigr], \label{eq:Theorem2Eq2}\\
\frac{d}{d\TT}&\log \tau\bigl(\longexactcohthreetheta\bigr) =\frac{d}{d\TT}\log T(\thetacomplex).
\label{eq:Theorem2Eq3}
\end{align}
Furthermore,
\begin{equation}\label{eq:1007111}
\TT\mapsto \log T(\thetacomplex)-\log \tau (\sH_\TT)
\end{equation}
is differentiable for $0\le \TT <\pi/2$. Here, $\sH_\TT$
stands for either $\longexactcohonetheta$ or $\longexactcohthreetheta$.
\end{theorem}

The proof of Theorem \ref{thm:Theorem2} will occupy the next Section \ref{s:GTP}.


\section{Gauge transforming the parametrized de Rham complex a la
Witten}\label{s:GTP}                                            
Consider the manifold $X$ as described in Section \ref{s:VMB}.
Recall that in the collar $W:=(-c,c)\times Y$ of $Y$ all structures
are assumed to be product. We introduce $W^{\cut}:= (-c,0]\times Y \coprod
[0,c)\times Y$. Furthermore, let $S: W^{\cut}\to W^{\cut}, (t,p)\mapsto (-t,p)$
be the reflexion map at $Y$. Finally, we introduce the map
\begin{equation}\label{eq:GTP1}
 T:\Omega^\bullet(W^{\cut};F)\longrightarrow \Omega^\bullet(W^{\cut};F),
 \quad T(\go_1,\go_2):= (S^*\go_2,-S^*\go_1).
\end{equation} 
$T$ is a skewadjoint operator in $L^2(W^{\cut},\Lambda^\bullet T^*W^{\cut}\otimes
F)$ with $T^2=-I$. Note furthermore, that $T$ commutes with the exterior
derivative $d$. We denote by $D^\TT$ (on $X^{\cut}$ resp. $W^{\cut}$)
the closed extension of the
exterior derivative with boundary conditions as in \Eqref{eq:1007096} along $Y$. 
More precisely, $D^\TT$ acts on the domain
\begin{multline}\label{eq:GTP2}
 \sD_\TT^j(W;F):=\bigl\{(\go_1,\go_2)\in \dom(d_{j,\max}) \bigm| \cost\cdot
 i_-^*\go_1=\sint\cdot i_+^* \go_2\bigr\}.    
\end{multline}

The operator family has varying domain. In order to obtain variation formulas
for functions of $D^\TT$ we will apply the method of gauge--transforming
$D^\TT$ onto a family with constant domain, \cf \eg \cite{DouWoj:ALE},
\cite{LesWoj:EIG}. 

We choose a cut--off function $\varphi\in\cinfz{(-c,c)\times Y}$
with $\varphi\equiv 1$ in a neighborhood of $\{0\}\times Y$ and which
satisfies $\varphi(-t,p)=\varphi(t,p), (t,p)\in (-c,c)\times Y$.
Then we introduce the gauge transformation
\begin{equation}
    \Phi_\TT:=e^{\TT\varphi T}=\cos(\TT \varphi) I+ \sin(\TT\varphi) \,T:\Omega^\bullet(W^{\cut};F)\longrightarrow
    \Omega^\bullet(W^{\cut};F).
\end{equation}
Since $e^{\TT\varphi(t,p) T}=1$ for $|t|$ sufficiently close to $c$, 
$\Phi_\TT$ extends in an obvious way to a unitary transformation of
$L^2(\Lambda^*T^*X^{\cut};F)$ which maps smooth forms to smooth forms. 

\begin{lemma} \label{l:GTP1} For $\TT, \TT'\in\R$ the operator
 $\Phi_{\TT}$ maps $\dom_{\TT'}^j(X;F)$ onto
 $\dom_{\TT+\TT'}^j(X;F)$, and accordingly $\dom_{\TT'}^j(W^{\cut};F)$ onto
 $\dom_{\TT+\TT'}^j(W^{\cut};F)$. Furthermore,
 \begin{equation}\label{eq:GTP4}
  \Phi_{\TT}^* D^{\TT+\TT'} \Phi_{\TT} = D^{\TT'} + \TT \ext(d\varphi) T.
   \end{equation}
 \end{lemma}
 \begin{proof} It obviously suffices to
 prove the Lemma for $W^{\cut}$. Consider $(\go_1,\go_2)\in
 \dom_{\TT'}^j(W^{\cut};F)$. Then
 \begin{align}\label{eq:GTP5}
    i_-^*\Phi_\TT (\go_1,\go_2)&= \cost\cdot i_-^*\go_1+\sint\cdot i_+^*\go_2, \\
    i_+^*\Phi_\TT (\go_1,\go_2)&= \cost\cdot i_-^* \go_2-\sint\cdot i_+^*\go_1.
   \end{align}
A direct calculation now shows
\begin{equation}
 \cos(\TT+\TT')\, i_-^* \Phi_\TT (\go_1,\go_2)= \sin(\TT+\TT')\, i_+^*
 \Phi_\TT(\go_1,\go_2),
\end{equation} 
proving the first claim.
The formula \Eqref{eq:GTP4} follows since $T$ commutes with exterior
differentiation.
\end{proof} 

Note that $D^{\pi/4}+\TT \,\ext(d\varphi) T$ is a deformed de Rham operator
acting on smooth differential forms on the smooth manifold $X$ (resp. $W$).
$T$ is not a differential operator. However, the reflection map $S$
allows to identify $(-c,0)\times Y$ with $(0,c)\times Y$ and hence
sections in a vector bundle $E$ over $(-c,0)\times Y\coprod (0,c)\times Y$ 
may be viewed as sections in the vector bundle $E\oplus S^*E$ over
$(0,c)\times Y$. Therefore, since $\supp(d\varphi)$ is compact in 
$(-c,0)\times Y\coprod (0,c)\times Y$, $T$ may be
viewed as a bundle endomorphism acting on the bundle
$(\Lambda^*T^*(0,c)\times Y)\otimes (F\oplus F)$. In particular employing
the classical interior parametric elliptic calculus, as \eg in \cite{Shu:POS},
we infer that the Laplacian corresponding to $D^{\pi/4}+\TT \ext(d\varphi) T$
has discrete dimension spectrum over any such compact neighborhood of
$\supp(d\varphi)$ which does have positive distance from ${\pm c}\times Y$. 

\newcommand{\tDeltaTT}{\tilde \Delta^\TT}
\newcommand{\tdTT}{\tilde D^\TT}
\newcommand{\dTT}{D^\TT}
From now on let 
\begin{equation}\label{eq:GTP6}
     \tilde D^\TT := D^{\pi/4} + \TT \ext(d\varphi) T
\end{equation}
with domain $\sD_{\TT=\pi/4}^\bullet(X;F)$
and $\tDeltaTT=(\tdTT)^*\tdTT + \tdTT(\tdTT)^*$ the corresponding
Laplacian. On the collar $(-c,c)\times Y$ the operator $\tDeltaTT$
is of the form $P+V$ as discussed in Subsection \ref{ss:EOGSA}, where
$P$ is the form Laplacian and $V=\tDeltaTT-\Delta$ is induced
by $\TT \ext(d\varphi) T$. The subset $\cK$ of \Eqref{eq:EOG2} is the
support of $d\varphi$. The operator $\tDeltaTT$ is now obtained as
in \Eqref{eq:EOG8} by gluing the domains of the form Laplacians
of the given de Rham complexes on $X^\pm$. Prop. \ref{p:EOG5} and Cor.
\ref{cor:EOG6} now give

\begin{theorem}\label{t:GTP2} The Hilbert complexes $\thetacomplex$ defined in
 \Eqref{eq:1007096} are Hilbert complexes with discrete dimension
 spectrum.
\end{theorem}

\begin{theorem}\label{t:GTP3} For $0<\TT<\pi/2$ the Hilbert complexes
 $\thetacomplex$ satisfy \textup{(1)--(4)} of Prop. \ref{p:GenDiffResult}.
 More precisely,
 \begin{equation}\label{eq:GTP7}
  \frac{d}{d\TT} H_T(\thetacomplex)
  =- t\frac{d}{dt} \frac{4}{\sin 2\TT} \sum_{j\ge 0} (-1)^j 
     \Tr\bigl( \gb_\TT e^{-t\Delta^\TT_j}\bigr)
 \end{equation}
and
\begin{equation}\label{eq:GTP8}
 \sum_{j\ge 0} (-1)^j \Tr\bigl( \gb_\TT e^{-t\Delta^\TT_j}\bigr)
   = \chi(X^+;F) - \sin^2\TT \cdot \chi(Y;F) + O(t^\infty),
 \end{equation}
as $t\to 0+$.
\end{theorem}

\subsection{Proof of Theorem \ref{t:GTP3}}    
Note that
\begin{equation}\label{eq:GTP10}
 \frac{d}{d\TT} \dpi = \ext(d\varphi) T = [d,\varphi T].
\end{equation}
Let us reiterate that although $[d,\varphi T]$ is strictly speaking
not a $0$th order differential operator it may be viewed as one
over $(\Lambda^*T^*(0,c)\times Y)\otimes (F\oplus F)$, which
implies that it lies in $\Opc^0(W^{\cut})\subset\Opc^0(X^{\cut})$.

We remind the reader of the definition of the closed and coclosed
Laplacians in \Eqref{eq:cl}, \eqref{eq:ccl}. We find
\begin{equation}\label{eq:GTP11}
 \begin{split}
  \frac{d}{d\TT}& \Tr\bigl(e^{-t\Deltaccl{p}}\bigr) =  \frac{d}{d\TT}
  \Tr\bigl(e^{-t\tDeltaccl{p}}\bigr)\\
  &= -t\Trbig{\bigl(\tdTT_p)^t \ext(d\varphi) T + (\ext(d\varphi) T)^t
  \tdTT_p\bigr) e^{-t\tDeltaccl{p}}   }\\
  &= -t\Trbig{\bigl(\dTT_p)^t \ext(d\varphi) T + (\ext(d\varphi) T)^t
  \dTT_p\bigr) e^{-t\Deltaccl{p}}   },\\
\end{split}
\end{equation}
where in the last line we have used that $\Phi_\TT$ commutes with
$\ext(d\varphi) T$.

Next let $(e_n)_{n\in\N}$ be an orthonormal basis of 
$\ran (D_p^\TT)^*$ consisting of eigenvectors of $\Deltaccl{p}$
to eigenvalues $\gl_n>0$.
Then $(\tilde e_n=\gl_n^{-1/2} e_n)_n$ is an orthonormal
basis of $\ran D_p^\TT$ consisting of eigenvectors of $\Deltacl{p+1}$
(\textit{cf.} \Eqref{eq:cl}, \eqref{eq:ccl} and thereafter).
\Eqref{eq:GTP11} gives
\begin{equation}\label{eq:GTP133}
 \begin{split}
  \frac{d}{d\TT} \Tr\bigl(e^{-t\Deltaccl{p}}\bigr) 
  =& -t \sum_n \scalar{(d_p^\TT)^t \ext(d\varphi) T
  e^{-t\Deltaccl{p}} e_n}{e_n}\\
  & -t \sum_n \scalar{e^{-t\Deltaccl{p}} e_n}{(d_p^\TT)^t \ext(d\varphi) T e_n}\\
  =& -2t \Re\Bigl(\sum_n \scalar{(d_p^\TT)^t \ext(d\varphi) T
  e^{-t\Deltaccl{p}} e_n}{e_n}\Bigr).
 \end{split}
\end{equation}
Stokes' Theorem and the boundary conditions will allow to rewrite
the individual summands of the last sum. To this end
let $\go,\eta\in \dom(\Delta_p^\TT)$. Then since $d\varphi$ is compactly
supported in the interior of $W^{\cut}$ we have
\begin{equation}\label{eq:GTP12}
 \begin{split}
 \scalar{(d_p^\TT)^t &\ext(d\varphi) T \go}{\eta}=
 \scalar{\ext(d\varphi)T\go}{d\eta}= \scalar{d\varphi \wedge T\go}{d\eta}\\
  &= \scalar{d(\varphi T\go)}{d\eta}- \scalar{\varphi T d\go}{d\eta}\\
  &= \int_{\pl X^{\cut}} T\go \wedge \tilde* d\eta+\scalar{\varphi T\go}{d^td\eta}
           -\scalar{\varphi T d\go}{d\eta}.
          \end{split}
\end{equation}
Here, $\tilde *$ denotes the natural isometry $\wedge^p T^*M\otimes
F\longrightarrow \wedge^{m-p}T^*M\otimes F^\dagger$.
In the last equality we have applied Stokes' Theorem on the manifold
with boundary $X^{\cut}$. Note that $\varphi T \go$ is a compactly
supported (locally of Sobolev class at least $2$) form on $X^{\cut}$.

The boundary of $X^{\cut}$ consists of two copies of $Y$ with
opposite orientations. To calculate the integral in the last
equation we orient $Y$ as the boundary of $X^+$. Then using
that $\go$ and $\eta$ satisfy the boundary conditions
\Eqref{eq:1007096} at $Y$ we find
\begin{equation}\label{eq:GTP13}\begin{split}
\int_{\pl\Xcut}T\go\wedge \tilde *d\eta&= \int_Y i_+^*(T\go\wedge\tilde *d\eta)
               -i_-^*(T\go\wedge\tilde *d\eta)\\
&=-\int_Y i_-^*\go\wedge i_+^*\tilde *d\eta+i_+^*\go\wedge i_-^*\tilde *d\eta\\
&=-\bigl(\tan\TT+\cot\TT\bigr)\int_Y i_+^*(\go\wedge \tilde *d\eta)\\
&=-\frac{2}{\sin2\TT} \bigl(\int_{X^+} d\go\wedge \tilde*d\eta + (-1)^{|\go|} \go
     \wedge d\tilde * d\eta \bigr)\\
&= -\frac{2}{\sin2\TT}
\bigl(\scalar{d\go}{d\eta}_{X^+}-\scalar{\go}{d^td\eta}_{X^+}\bigr).
		\end{split}
\end{equation}
Here $\scalar{\cdot}{\cdot}_{X^+}$ denotes the $L^2$--scalar product of forms over
$X^+$. 

Plugging into \Eqref{eq:GTP12} gives
\begin{multline}\label{eq:GTP131}
 \scalar{(d_p^\TT)^t \ext(d\varphi) T \go}{\eta}\\
   = \scalar{\varphi T \go}{d^td \eta} - \scalar{\varphi T d\go}{d \eta}
     - \frac{2}{\sin2\TT}
    \bigl(\scalar{d\go}{d\eta}_{X^+}-\scalar{\go}{d^td\eta}_{X^+}\bigr).
\end{multline}
Similarly,
\begin{multline}\label{eq:GTP132}
 \scalar{(\ext(d\varphi) T)^t D_p^\TT \go}{\eta}\\
   = \scalar{d^td \go}{\varphi T \eta} - \scalar{d \go}{\varphi T d\eta}
     - \frac{2}{\sin2\TT}
    \bigl(\scalar{d\go}{d\eta}_{X^+}-\scalar{\go}{d^td\eta}_{X^+}\bigr).
\end{multline}

We now apply \Eqref{eq:GTP131} to the summands on the right of
\Eqref{eq:GTP133} and find using \Eqref{eq:clccl}
\begin{equation}\label{eq:GTP134}
 \begin{split}
  \scalar{&(d_p^\TT)^t \ext(d\varphi) T e^{-t\Deltaccl{p}}
  e_n}{e_n}\\
  =& \scalar{\varphi T e^{-t\Deltaccl{p}}\Deltaccl{p}e_n}{e_n}-
   \scalar{\varphi T e^{-t\Deltacl{p+1}}\Deltacl{p+1}\tilde e_n}{\tilde e_n}\\
   & - \frac{2}{\sin2\TT}\bigl(\scalar{\gb_\TT
   e^{-t\Deltacl{p+1}}\Deltacl{p+1} \tilde e_n}{\tilde e_n}
      - \scalar{\gb_\TT
   e^{-t\Deltaccl{p}}\Deltaccl{p} e_n}{e_n}\bigr),
 \end{split}
\end{equation}
and summing over $n$ gives 
\begin{equation}\label{eq:GTP14}
 \begin{split}
  \frac{d}{d\TT}& \Tr\bigl(e^{-t\Deltaccl{p}}\bigr) \\
  =& - 2t\Re  \Bigl(\Trbig{\varphi T e^{-t \Deltaccl{p}}\Deltaccl{p}} -
  \Trbig{\varphi T e^{-t \Deltacl{p+1}} \Deltacl{p+1}}\Bigr)\\
  &\quad +\frac{4t}{\sin2\TT}\Re \Bigl( \Trbig{\gb_\TT \Deltacl{p+1}
  e^{-t\Deltacl{p+1} }}- \Trbig{\gb_\TT \Deltaccl{p} e^{-t\Deltaccl{p}} }
  \Bigr)
  \\
  =& 2t \frac{d}{dt}\frac{2}{\sin2\TT} \Bigl( \Trbig{\gb_\TT
         e^{-t\Deltaccl{p}} }- \Trbig{\gb_\TT  e^{-t\Deltacl{p+1} }} \Bigr).
\end{split}
\end{equation} 
Here we have used that since $\varphi T$ is skew--adjoint $\Tr(\varphi T A)$
is purely imaginary
for every selfadjoint trace class operator $A$ and similarly that
since $\gb_\TT$ is selfadjoint that $\Tr(\gb_\TT A)$ is real.
Consequently using \Eqref{eq:HTccl}
\begin{equation}
 \begin{split}
  \frac{d}{d\TT} &H_T(\thetacomplex) = \frac{d}{d\TT} \sum_{j\ge 0} (-1)^{j+1}
  \Trbig{e^{-t \Deltaccl{j} }}\\
  =& -2t \frac{d}{dt}\frac{2}{\sin2\TT}\sum_{j\ge 0} (-1)^j \Trbig{\gb_\TT
         e^{-t\Delta_j^\TT}}.
        \end{split}
\end{equation}        

\newcommand{\Deltar}[1]{\Delta_{#1}^{r}}
\newcommand{\Deltaa}[1]{\Delta_{#1}^{a}}
Finally, for calculating the asymptotic expansion \Eqref{eq:GTP8} as $t\to 0+$
we may again invoke our Corollary \ref{eq:EOG6}. The asymptotic expansion
\Eqref{eq:GTP8} on $\Xcut$ differs from the corresponding expansion for
the double $-X^+\coprod X^+$ by an error term $O(t^\infty)$; here $-X^+$ stands
for $X^+$ with the opposite orientation. However,
on the double $-X^+\coprod X^+$ we may write down the heat kernel
for $\Delta_p^\TT$ explicitly in terms of the heat kernels for $\Delta_p$
with relative and absolute boundary conditions at $Y$ \cite[(2.118) p.~60]{Vis:GRS}. Namely, let
$\Deltar{p}, \Deltaa{p}$ be the Laplacians of the relative and
absolute de Rham complexes on $X^+$ as in
\Eqref{eq:AbsRelBoundaryCondition} and denote by $E_t^{p,r/a}$ their corresponding heat kernels.
Let $S$ be the reflection map which interchanges the two copies of $X^+$ in
$-X^+\coprod X^+$.
Its restriction to $W$ is the reflection map $S$ defined before \Eqref{eq:GTP1}
and hence denoting it by the same letter is justified.

Finally, let $E_t^p(x,y)$ be the heat kernel of $\Delta_p^{\pi/4}$ on
$-X^+\coprod X^+$, \ie the Laplacian with continuous transmission boundary
conditions at $Y$. Then the absolute/relative heat kernels are
given in terms of  $E_t^p$ by
\begin{equation}
   E_t^{p,a}= (E_t^p+ S^*\circ E_t^p)\restriction X^+; \quad E_t^{p,r}=
   (E_t^p- S^*\circ E_t^p)\restriction X^+.
  \end{equation}     
More generally, we put for $x,y\in -X^+\coprod X^+$:
\begin{equation}
   E_t^{p,\TT}(x,y):=\begin{cases} E_t^p(x,y) + \cos(2\TT) (S^*\circ
    E_t^p)(x,y), & \text{ if } x,y\in X^+,\\
    \sin(2\TT) E_t^p(x,y), & \text{ if } x\in (-X^+), y\in X^+.
   \end{cases}
  \end{equation}   
One immediately checks that $E_t^{p,\TT}$ is the heat kernel
of $\Delta_p^\TT$ on $-X^+\coprod X^+$.
Consequently
\begin{equation}\begin{split}
 \Trbig{\gb_\TT &e^{-t\Delta_p^\TT}}= \Trbig{\gb_\TT E_t^p} + \cos(2\TT)
 \Trbig{S^*\circ E_t^p}\\
 &= \cos^2(\TT) \Tr(E_t^{p,a}) +\sin^2(\TT) \Tr(E_t^{p,r}).
\end{split}
\end{equation}
Since in view of our Standing Assumptions \ref{ss:VMBSA}
the complexes $\sD^\bullet(X^+,Y;F)$ and $\sD^\bullet(X^+;F)$
are Fredholm complexes the McKean-Singer formula \Eqref{eq:McKS}
holds and hence taking alternating sums yields
\begin{equation}
 \begin{split}
  \sum_{j\ge 0} (-1)^j \Trbig{\gb_\TT e^{-t \Delta_j^\TT}}&=
  \cos^2\TT\cdot\chi(X^+;F)+\sin^2\TT\cdot \chi(X^+,Y;F)\\
    &= \chi(X^+;F)-\sin^2\TT\cdot \chi(Y;F)
   \end{split}
\end{equation}   
and the proof of \Eqref{eq:GTP8} is complete.
In the last equality we have used that $\chi(X^+;F)=\chi(X^+,Y;F)+\chi(Y;F)$;
this formula follows from the exact sequence \Eqref{eq:ExactSeq3}.\hfill\qed

\subsection{Proof of Theorem \ref{thm:Theorem2}}
\subsubsection{Proof of \eqref{eq:Theorem2Eq1}}
Combining Prop.
\ref{p:GenDiffResult} and Theorem \ref{t:GTP3} we find
\begin{equation}\label{eq:GTP9}
 \begin{split}
  \frac{d}{d\TT} &\log T(\thetacomplex) \\
    = &- \frac 12 \frac{-4}{\sin 2\TT}\; \Bigl(\chi(X^+;F)-\sin^2\TT\; \chi(Y;F)
    \Bigr)\\
     & + \frac 12 \frac{-4}{\sin 2\TT} \Tr\Bigl(\sum_{j\ge 0} (-1)^j
      \gb_\TT\restriction H_\TT^j(X;F)\Bigr)\\
      =& \frac{2}{\sin 2\TT}\Bigl[- \altsum \trbetatheta + \chi(X^+;F)\Bigr]-
      \tan\TT\cdot \chi(Y;F)
 \end{split}
\end{equation}
which is the right hand side of \Eqref{eq:Theorem2Eq1}.\hfill\qed

\subsubsection{Proof of \eqref{eq:Theorem2Eq2} and \eqref{eq:Theorem2Eq3}}
Let $0<\TT,\TT'<\pi/2$ and consider the following commutative diagram,
\cf \Eqref{eq:ExactSeq1p}
\begin{equation}\label{eq:ExactSeq1Comp}
\xymatrix{
    0 \ar[r]  &\leftrelcomplex \ar[r]^-{\ga_\TT}\ar[d]^{\id} 
           &\thetacomplex \ar[r]^-{\beta_\TT}\ar[d]^{\phi_{\TT,\TT'}}
     &\rightabscomplex \ar[r]\ar[d]^{\phi_{\TT,\TT'}^+} &  0\\
   0 \ar[r]  &\leftrelcomplex \ar[r]^-{\ga_{\TT'}} 
           &\thetavarcomplex{\TT'} \ar[r]^-{\beta_{\TT'}}  
     &\rightabscomplex \ar[r] &  0,
     }
\end{equation}
where $\phi_{\TT,\TT'}(\go_1,\go_2)=(\go_1,\frac{\tan\TT}{\tan\TT'} \go_2)$
resp.  $\phi_{\TT,\TT'}^+(\go_2)=(\frac{\tan\TT}{\tan\TT'} \go_2)$.
$\phi_{\TT,\TT'}, \phi_{\TT,\TT'}^+$ are Hilbert complex isomorphisms and the
diagram \eqref{eq:ExactSeq1Comp} commutes. Hence we obtain 
a cochain isomorphism between the long exact cohomology sequences
of the upper and lower horizontal exact sequences ($F$ omitted to save horizontal space):
\begin{equation}\label{eq:ExactSeq1CompLongExactCoh}
\xymatrix{
    \ldots   H^k(X^-,Y) \ar[r]^-{\ga_{\TT,*}}\ar[d]^{\id} 
           &H^k_\TT(X) \ar[r]^-{\beta_{\TT,*}}\ar[d]^{\phi_{\TT,\TT',*}}
     &H^k(X^-) \ar[r]^-{\delta_\TT}\ar[d]^{\phi_{\TT,\TT',*}^+} 
     & H^{k+1}(X^-,Y)\ar[d]^{\id}\ldots\\
     \ldots   H^k(X^-,Y) \ar[r]^-{\ga_{\TT',*}} 
           &H^k_{\TT'}(X) \ar[r]^-{\beta_{\TT',*}}
     &H^k(X^-) \ar[r]^-{\delta_{\TT'}} 
     & H^{k+1}(X^-,Y)\ldots
     }
\end{equation}

Let $e_1,\ldots,e_r$ be an orthonormal basis of $H^k_\theta(X;F)$.
Then 
\begin{equation}\label{eq:100}  
   \Det \bigl(\phi_{\TT,\TT',*}^k\bigr)^2= \Det 
      \bigl(
      \scalar{\phi_{\TT,\TT',*}e_i}{\phi_{\TT,\TT',*}e_j}_{i,j=1}^r\bigr),
\end{equation}
hence
\begin{align}  
 \dThetaPrimeAtTheta 
         &\log  \Det \bigl(\phi_{\TT,\TT',*}^k\bigr)^2\nonumber\\
         =&\Tr\Bigl(\bigl(\scalar{\dThetaPrimeAtTheta \phi_{\TT,\TT',*} e_i}{e_j}+
	      \scalar{e_i}{\dThetaPrimeAtTheta \phi_{\TT,\TT',*}
       e_j}\bigr)_{i,j=1}^r\Bigr)\nonumber\\
        =& - 2 \frac{2}{\sin 2\TT} \sum_{j=0}^r \scalar{\gb_\TT e_j}{e_j}
        = - 2 \frac{2}{\sin 2\TT} \trbetathetak,\label{eq:101}   
\end{align}
see \Eqref{eq:1007097}.
Furthermore, since $\phi_{\TT,\TT'}^+$ is multiplication by $\frac{\tan\TT}{\tan\TT'}$
we have
\begin{equation}\label{eq:103}  
    \Det\bigl( \phi_{\TT,\TT',*}^+\bigr)= \bigl(\frac{\tan\TT}{\tan\TT'}\bigr)^{\chi(X^+;F)}
\end{equation}
and hence
\begin{equation}\label{eq:104}  
    \dThetaPrimeAtTheta \log  \Det\bigl( \phi_{\TT,\TT',*}^+\bigr)^2= -
             2 \frac{2}{\sin 2\TT}\; \chi(X^+;F).
\end{equation}
By Lemma \plref{l:OMZ5} we have
\begin{equation}\label{eq:102}  
 \begin{split}
    \log\tau\bigl(&\longexactcohonethetavar{\TT'}\bigr)-\log\tau\bigl(\longexactcohonetheta\bigr)\\
       & = \frac 12 \log \Det \bigl( \phi_{\TT,\TT',*}\bigr)^2 
      -\frac 12 \log \Det \bigl( \phi_{\TT,\TT',*}^+\bigr)^2 ;
  \end{split}
\end{equation}
combined with \Eqref{eq:101} and \Eqref{eq:104} we therefore find \eqref{eq:Theorem2Eq2}. 

That the left hand side of \eqref{eq:Theorem2Eq3} equals the right hand
side of \eqref{eq:Theorem2Eq1} is proved analogously. One just has to replace
the commutative diagram \eqref{eq:ExactSeq1Comp} by 
\begin{equation}\label{eq:ExactSeq3Comp}
\xymatrix{
    0 \ar[r]  &\leftrelcomplex\oplus \rightrelcomplex 
                \ar[r]^-{\gamma_++\gamma_-}\ar[d]^{\id\oplus \phi_{\TT,\TT'}^+} 
           &\thetacomplex \ar[r]^{r_\TT}\ar[d]^{\phi_{\TT,\TT'}}
     &\bdycomplex \ar[r]\ar[d]^{\psi_{\TT,\TT'}} &  0\\
   0 \ar[r]  &\leftrelcomplex\oplus \rightrelcomplex
   \ar[r]^-{\gamma_++\gamma_-} 
           &\thetavarcomplex{\TT'} \ar[r]^{r_{\TT'}}  
     &\bdycomplex \ar[r] &  0,
     }
\end{equation}
where $\psi_{\TT,\TT'}(\go)=\frac{\sin\TT}{\sin\TT'}\go$.
See also \Eqref{eq:HA20} and thereafter.\hfill\qed

\subsubsection{Proof of the differentiability of \Eqref{eq:1007111} at
$0$}\label{ss:diff}
The problem is that the dimensions of the cohomology groups 
$H^j_\TT(X;F)$ may jump at $0$; note that the isomorphism
$\phi_{\TT,\TT'}$ defined after \Eqref{eq:ExactSeq1CompLongExactCoh}
between $\thetacomplex$ and $\thetavarcomplex{\TT'}$
is defined only for $0<\TT, \TT'<\pi/2$. 
By our Standing Assumptions \ref{ss:VMBSA}, \cf also Subsection \ref{ss:IBC},
$\leftrelcomplex$ and $\rightabscomplex$ are Hilbert complexes with
discrete dimension spectrum. Hence we may choose $a>0$ such that $a$ is smaller than the smallest
nonzero eigenvalues of the Laplacians of $\leftrelcomplex$ and
$\rightrelcomplex$. Furthermore, we denote by $\Pi^p_\TT$ the orthogonal
projection onto 
\begin{equation}\label{eq:diff1}
   H_{\TT,a}^p(X;F):=\bigoplus_{0\le \gl<a} \ker(\Delta_p^\TT-\gl).
\end{equation}
Since for $\TT=0$ the complex $\thetacomplex$ is canonically
isomorphic to the direct sum $\leftrelcomplex\oplus\rightabscomplex$ and
since the gauge--transformed Laplacian $\tilde\Delta^\TT$ of $\thetacomplex$ in view
of \Eqref{eq:GTP6}
certainly depends smoothly on $\TT$ there exists a $\TT_0>0$
such that the projection $\Pi_\TT^p$ depends smoothly on $\TT$ for
$0\le \TT<\TT_0$. In particular $\rank \Pi_\TT^p=\dim H^p_{\TT=0}(X;F)$ is
constant for $0\le \TT<\TT_0$.

$\bigl(H_{\TT,a}^\bullet(X;F),d\bigr)$ is a finite-dimensional Hilbert complex
and the orthogonal projections $\Pi_\TT^p$ give rise to a natural
orthogonal decomposition of Hilbert complexes
\begin{equation}\label{eq:diff2}
   \thetacomplex=: \bigl(H_{\TT,a}^\bullet(X;F),d\bigr) \oplus 
   \thetavarcomplex{{\TT,a}}.
\end{equation} 
By construction of $\Pi_a^p$ we have
\begin{equation}\label{eq:diff3}
   \log T(\thetacomplex)= \log \tau \bigl(H_{\TT,a}^\bullet(X;F),d\bigr)
      +   \log T(\thetavarcomplex{{\TT,a}}),
\end{equation}
and $\TT\mapsto \log T(\thetavarcomplex{{\TT,a}})$ is differentiable
for $0\le \TT<\TT_0$.

Since surjectivity is an open condition we conclude that the
sequence
\begin{equation}\label{eq:diff4}
\xymatrix{
0 \ar[r]   &H^*(X^-,Y;F) \ar@{^{(}->}[r]^-{\ga_\TT}  &H_{\TT,a}^*(X;F) \ar[r]^{\beta_\TT}
     &H^*(X^+;F) \ar[r] &  0,
}
\end{equation}
is exact for $0\le \TT<\TT_1\le \TT_0$. Here, $\ga_\TT$ is defined in the
obvious way while
\begin{equation}\label{eq:diff5}
   \gb_\TT:= \text{ orthogonal projection onto $H^*(X^+;F)$ of
   $\go\restriction X^+$}.
  \end{equation}
Note that the differentials of the left and right complexes vanish and
hence so do their torsions. The space of harmonics of the middle
complex equals the space of harmonics of the complex $\thetacomplex$
and hence the cohomology of the middle complex is (isometrically)
isomorphic to the cohomology of $\thetacomplex$. 
One immediately checks that the long exact cohomology sequence of
\Eqref{eq:diff4} is exactly the exact cohomology sequence
$\longexactcohone$. Hence Prop. \plref{p:OMZ6} 
yields
\begin{multline}
 \log\tau(H_{\TT,a}^*(X;F))=\log \tau\bigl(\longexactcohone\bigr)\\
 -\sum_{p\ge 0} \log \tau\bigl(
 0\rightarrow H^p(X^-,Y;F)\stackrel{\ga_\TT}{\rightarrow} H_{\TT,a}^p(X;F)
   \stackrel{\gb_\TT}{\rightarrow} H^p(X^+;F)\rightarrow 0\bigr).
 \end{multline}
This shows the differentiability of the difference
$\log\tau(H_{\TT,a}^*(X;F))-\log \tau\bigl(\longexactcohone\bigr)$
at $\TT=0$. In view of \Eqref{eq:diff3} the claim is proved.
\hfill\qed


\section{The gluing formula}\label{s:GF}                     
We can now state and prove the main result of this paper. The Standing
Assumptions \ref{ss:VMBSA} are still in effect. Furthermore, we will
use freely the notation introduced in Subsection \ref{ss:Main}.

\begin{theorem}\label{thm:Theorem1}
For the analytic torsions of the Hilbert complexes $\genrelcomplex$, 
$\genabscomplex$, $\totalcomplex$ we have the following formulas:
\begin{align}
\log T(\totalcomplex) =& \log T(\leftrelcomplex)+\log T(\rightabscomplex) \label{eq:Theorem1Eq1}\\
                       &+ \log \tau\bigl( \longexactcohone\bigr) \nonumber 
		       - \frac 12 \log 2 \cdot \chi(Y;F), \\
\log T(\leftabscomplex) =& \log T(\leftrelcomplex)+\log T(\bdycomplex)  \label{eq:Theorem1Eq2}\\
                         & + \log \tau\bigl(\longexactcohtwominus\bigr),\nonumber\\
 \log T(\totalcomplex) =& \log T(\leftrelcomplex)+\log T(\rightrelcomplex) \label{eq:Theorem1Eq3}\\
                       &+ \log \tau\bigl( \longexactcohthree\bigr). \nonumber
\end{align}
\end{theorem}

\subsection{Proof of Theorem \ref{thm:Theorem1}} In the course of the proof
we will make heavy use of Theorem \ref{thm:Theorem2}.

\subsubsection{Proof of \eqref{eq:Theorem1Eq1}}
As noted after \Eqref{eq:1007096} we have for $\TT=0$ that
$\thetavarcomplex{\TT=0}=\leftrelcomplex\oplus \rightabscomplex$
and that for $\TT=\pi/4$ the complexes $\thetavarcomplex{\TT=\pi/4}$
and $\totalcomplex$ are isometric. Hence we have
\begin{equation}
\begin{split}
 \log T&(\totalcomplex) - \log T(\leftrelcomplex)-\log T(\rightabscomplex)\\
      =& \log T(\thetavarcomplex{\pi/4})- \log T(\thetavarcomplex{\TT=0})\\
      =&\log T(\thetavarcomplex{\pi/4})-\log \tau(\longexactcohonethetavar{\pi/4})\\
       & -\log
       T(\thetavarcomplex{\TT=0})+\log\tau\bigl(\longexactcohonethetavar{\TT=0}\bigr)\\
       &   + \log\tau\bigl(\longexactcohonethetavar{\pi/4}\bigr).
\end{split}
\end{equation}
Recall that for $\TT=0$ the complex $\thetavarcomplex{\TT=0}$ is just the
direct sum complex $\sD^\bul(X^-,Y;F)\oplus \sD^\bul(X^+;F)$ and hence
$\log\tau\bigl(\longexactcohonethetavar{\TT=0}\bigr)=0$ (see also the
sentence after \Eqref{eq:OMZ26}). Furthermore, 
$\log\tau\bigl(\longexactcohonethetavar{\pi/4}\bigr)=\log\tau\bigl(\longexactcohone\bigr)$ hence
by Theorem \plref{thm:Theorem2}
\begin{equation}\label{eq:1007112}
  \begin{split}
    \ldots=&\int_0^{\pi/4}-\tan\TT d\TT\; \chi(Y;F)
    +\log\tau\bigl(\longexactcohone\bigr)\\
          =& -\frac 12\log 2\; \chi(Y;F)+\log\tau\bigl(\longexactcohone\bigr),
  \end{split}
\end{equation}
and we arrive at \Eqref{eq:Theorem1Eq1}.\hfill\qed

\subsubsection{Proof of \eqref{eq:Theorem1Eq2}}
Consider $\eps>0$ and apply the proved \Eqref{eq:Theorem1Eq1} to the manifold
$X_\eps^-:=X^-\cup_Y [0,\eps]\times Y$. Then
\begin{multline}\label{eq:1007115}
   \log T(\sD^\bullet(X_\eps^-;F)) = \log T(\leftrelcomplex) 
                 +\log T(\sD^\bullet([0,\eps]\times Y;F))\\
               -\frac 12 \log 2\; \chi(Y;F)
	       +\log \tau\bigl(\sH((X_\eps^-,Y),X_\eps^-,[0,\eps]\times Y;F)\bigr).
\end{multline}
For the cylinder $[0,\eps]\times Y$ it is well--known 
(it also follows easily from Proposition \plref{p:1007116}) that
\begin{align}\label{eq:1007116}
   \chi([0,\eps]\times Y;F)=&\chi(Y;F)=\chi(Y) \rank F, \\
  \log T(\sD^\bullet([0,\eps]\times Y;F)) =&\log T(\bdycomplex)\, \chi([0,\eps])
            + \chi(Y;F) \log T(\sD^\bullet([0,\eps])\nonumber\\
	    =& \log T(\bdycomplex) +\frac 12 \log(2\eps) \; \chi(Y;F).\label{eq:1007116a}
  \end{align}
Hence
\begin{multline}\label{eq:421}
    \log T(\sD^\bullet(X_\eps^-;F)) = \log T(\leftrelcomplex) 
                 +\log T(\sD^\bullet(Y;F))\\
                 +\frac12 \log\eps\;\chi(Y;F)
	       +\log \tau(\sH((X_\eps^-,Y),X_\eps^-,[0,\eps]\times Y;F)).
      \end{multline}       
 
In the sequel we will, to save some space,
omit the bundle $F$ from the notation in commutative diagrams. Our first
commutative diagram is
\begin{equation}
\xymatrix{
&\ldots  \ar[r] & H^k(X^-,Y) \ar[r]^{\ga_{-,*}}\ar[d]^{\psi_\eps^*}
&H^k(X_\eps^-) \ar[r]^{\gb_{*}}\ar[d]^{\psi_\eps^*} 
                &H^k([0,\eps]\times Y)\ar[r]\ar[d]^{\chi_\eps^*}&\ldots\\
&\ldots  \ar[r] & H^k(X^-,Y) \ar[r] &H^k(X^-) \ar[r] &H^k(Y)\ar[r]&\ldots
}
\end{equation}
The first row is the long exact cohomology sequence of \Eqref{eq:ExactSeq1}
for $X_\eps^-=X^-\cup_Y [0,\eps]\times Y$ instead of $X$; the second row
is the long exact cohomology sequence of \Eqref{eq:ExactSeq1} for
$X=X^-\cup_Y X^+$.
$\psi_\eps$ is a diffeomorphism $X^-\to X_\eps^-$ obtained as follows:
choose a diffeomorphism $f:[-c,0]\to [-c,\eps]$ such that $f(x)=x$ for
$x$ near $-c$ and $f(x)=x+\eps$ for $x$ near $0$. Then $\psi_\eps$ is obtained
by patching the identity on $X^-\setminus [-c,0]\times Y$ and $f\times\id_Y$.
Furthermore $\chi_\eps:Y\to [0,\eps]\times Y, p\mapsto (\eps,p)$.

For a harmonic form $\go\in \dom(d_{k,\max})\cap\dom((d_{k-1,\max})^*)
\subset \Omega^k([0,\eps]\times Y;F)$
one has $\go=\pi^* \chi_\eps^*(\go)$ ($\pi:[0,\eps]\times Y\to Y$ the projection)
and thus
\begin{equation}\label{eq:1007127}
  \int_{[0,\eps]\times Y} \go\wedge\tilde *\go 
  =\eps \int_Y \chi_\eps^* \go \wedge \tilde* \chi_\eps^* \go.
\end{equation}
Therefore the determinant (in the sense of \Eqref{eq:Det}) 
of $\chi_\eps^*$ on the cohomology is given by $\eps^{-\frac 12 \chi(Y;F)}$.
Consequently by Lemma \ref{l:OMZ5}
\begin{equation}\label{eq:1007128}
\begin{split}
  \log \tau\bigl(&\sH((X_\eps^-,Y),X_\eps^-,[0,\eps]\times Y;F)\bigr)\\
  =& \log\tau \bigl(\longexactcohtwominus\bigr) -\frac 12\; \log\eps\; \chi(Y;F)\\
   &+ \log \Det \bigl(\psi_\eps^*:H^*(X^-,Y;F)\to H^*(X^-,Y;F)\bigr) \\
   &  -\log \Det \bigl(\psi_\eps^*:H^*(X_\eps^-;F) \to H^*(X^-;F)\bigr)
\end{split}
\end{equation}
Summing up \Eqref{eq:421}, \eqref{eq:1007128}
\begin{equation}
  \begin{split}
   \log T(\sD^\bullet(X_\eps^-;F)) =& \log T(\leftrelcomplex) +\log T(\bdycomplex)\\
      &+ \log\tau\bigl(\longexactcohtwominus\bigr) \\
      &  +\log \Det (\psi_\eps^*:H^*(X^-,Y;F)\to H^*(X^-,Y;F)) \\
   &  -\log \Det (\psi_\eps^*: H^*(X_\eps^-;F)\to H^*(X^-;F)).
\end{split}
\end{equation}
As $\eps\to 0$ the determinants of $\psi_\eps^*:H^*(X^-,Y;F)\to H^*(X^-,Y;F))$
resp. $\psi_\eps^*: H^*(X_\eps^-;F)\to H^*(X^-;F))$ tend to $1$ and
we obtain \eqref{eq:Theorem1Eq2}.
\hfill \qed

\subsubsection{Proof of \eqref{eq:Theorem1Eq3}}
We note first that $\tau(\longexactcohthreetheta)_{\big | \TT=0}=
\tau(\longexactcohtwoplus)$, hence by 
\eqref{eq:Theorem2Eq3} and \eqref{eq:1007111}
\begin{equation}\label{eq:1007129}
  \begin{split}
    \log T&(\thetavarcomplex{\TT=\pi/4}) 
               -\log \tau(\longexactcohthreetheta)_{\big | \TT=\pi/4}\\
      =&\log T(\thetavarcomplex{\TT=0}) 
                 -\log \tau(\longexactcohthreetheta)_{\big | \TT=0}\\
      =& \log T(\leftrelcomplex)+\log T(\rightabscomplex) \\
       & -   \log \tau\bigl(\longexactcohtwoplus\bigr)\\
      =& \log T(\leftrelcomplex)+\log T(\rightrelcomplex) + \log T(Y;F),\\ 
         \end{split}
\end{equation}
where in the last equality we have used the proved identity
\eqref{eq:Theorem1Eq2}. \hfill\qed


\appendix
\section{The homological algebra gluing formula}\label{s:HAG} 

We present here the analogues of Theorem \ref{thm:Theorem1} and 
\ref{thm:Theorem2} for finite-dimensional Hilbert-complexes. This
applies,\eg to the cochain complexes of a triangulation twisted
by a unitary representation of the fundamental group, \textit{cf.}, \eg
\cite[Sec.~1]{Mul:ATR}.

Let $(C^*_j,d^j), j=1,2$, be finite-dimensional Hilbert complexes.
Let $(B^*,d)$ be another such Hilbert complex and assume that we are
given \emph{surjective} homomorphisms of cochain complexes
\begin{equation}
 r_j: (C_j,d^j)\longrightarrow (B,d),\quad j=1,2.\label{eq:HA1}
\end{equation}
We denote by $C_{j,r}\subset C_j$ the kernel of $r_j$,
by $\ga:C_1\rightarrow C_1\oplus C_2$ the inclusion and by
$\gb:C_1\oplus C_2\rightarrow C_2$ the projection onto the
second factor.

For $\TT\in\R$ we define the following homological algebra
analogue of the complex $\thetacomplex$, \cf \Eqref{eq:1007096},
by putting
\begin{multline}\label{eq:HA2}
    \bigl(C_1\oplus_\TT C_2\bigr)^j
    :=\bigl\{(\xi_1,\xi_2)\in C_1^j\oplus C_2^j\bigm|
	    \cost\cdot r_1 \xi_1=\sint\cdot r_2 \xi_2\bigr\}.    
\end{multline}
$(\cTT,d=d^1\oplus d^2)$ is a subcomplex of $(C_1\oplus C_2,d^1\oplus d^2)$.
For $\TT=0$ we have $\cTT=C_{1,r}\oplus C_2$ and for $\TT=\pi/4$
we have a homological algebra analogue of the complex $\thetacomplex$.

Furthermore, we have the following analogues of the exact sequences
\Eqref{eq:ExactSeq2}, \eqref{eq:ExactSeq1p}, \eqref{eq:ExactSeq3p}
(note that the exact sequences \Eqref{eq:ExactSeq1}, 
\eqref{eq:ExactSeq3} are special cases of the exact sequences
\Eqref{eq:ExactSeq1p}, \eqref{eq:ExactSeq3p}):
\begin{equation}\label{eq:HA3}
 0 \longrightarrow C_{j,r} \stackrel{\gamma_j}{\longrightarrow}  
   C_j \stackrel{r_j}{\longrightarrow} B \longrightarrow  0,
\end{equation}
\begin{equation}\label{eq:HA4}
 0 \longrightarrow \conerel \stackrel{\ga_\TT}{\longrightarrow}\cTT
 \stackrel{\beta_\TT}{\longrightarrow} C_2 \longrightarrow  0,
\end{equation}
\begin{equation}\label{eq:HA5}
 0 \longrightarrow   C_{1,r}\oplus C_{2,r}  
 \stackrel{\gamma_1+\gamma_2}{\longrightarrow}
 \cTT \stackrel{r_\TT}{\longrightarrow} B \longrightarrow  0.
\end{equation}
Here, $\gamma_j$ is the natural inclusion, $\gb_\TT=\gb\restriction\cTT$,
$\ga_\TT(\xi)=(\xi,0)$, and $r_\TT(\xi_1,\xi_2)=\sint\cdot
r_1\,\xi_1+\cost\cdot r_2\xi_2$.
Denote by $\cH(C_{j,r},C_j,B), \lectt, \lecrr$ the long exact cohomology
sequences of \Eqref{eq:HA3}, \eqref{eq:HA4}, \eqref{eq:HA5}, resp.

Since all complexes are finite-dimensional we have Lemma \ref{l:OMZ5}
and Prop. \ref{p:OMZ6} at our disposal. The latter applied to
\Eqref{eq:HA3} immediately gives the analogue of \Eqref{eq:Theorem1Eq2}
\begin{equation}\label{eq:HA11}
 \log \tau(C_1) = \log \tau(C_{1,r}) +\log\tau(B)+\log\tau\bigl(\lecrel\bigr).
\end{equation} 
The other claims of Theorem \ref{thm:Theorem1} and \ref{thm:Theorem2}
have exact counterparts in this context as summarized in the following:

\begin{theorem}\label{t:HA1} 
\textup{1. }
The functions $\TT\mapsto \log \tau(\cTT)$,
$\log \tau\bigl(\lectt\bigr)$, $\log \tau\bigl(\lecrr\bigr)$
are differentiable for $0<\TT<\pi/2$ . Moreover, for $0<\TT<\pi/2$
\begin{multline}
\frac{d}{d\TT}\log \tau(\cTT) =\frac{2}{\sin 2\TT}\Bigl[- \altsum \HAtrbetatheta +\\
        +\altsum  \HACCtrbetatheta \Bigr] \label{eq:HA7}, 
\end{multline}        
\begin{align}
\frac{d}{d\TT}&\log \tau(\lectt) =\nonumber\\
      &  =\frac{2}{\sin 2\TT}\Bigl[- \altsum \HAtrbetatheta + \chi(C_2)\Bigr],
      \label{eq:HA8}\\
\frac{d}{d\TT}&\log \tau\bigl(\lecrr\bigr) =\nonumber\\
      &  =\frac{2}{\sin 2\TT}\Bigl[- \altsum \HAtrbetatheta + \chi(C_2)\Bigr]
      -\tan\TT\; \chi(B).
      \label{eq:HA9}
\end{align}
Furthermore,
\begin{equation}\label{eq:HA10}
\TT\mapsto \log T(\cTT)-\log \tau \bigl(\cH_\TT\bigr)
\end{equation}
is differentiable for $0\le \TT <\pi/2$. Here, $\cH_\TT$
stands for either $\lectt$ or $\lecrr$.

\textup{2. } Under the additional assumption that the
$r_j$ are \emph{partial isometries} we have:
\begin{equation}\label{eq:HA10a}
 \frac{d}{d\TT}\log \tau(\cTT) = \frac{d}{d\TT} \log \tau(\lecrr),
\end{equation} 
and
\begin{align}
 \log\tau(\cTT)=& \log\tau(C_{1,r})+\log\tau(C_{2,r})\label{eq:HA12}\\
       &+\log\tau(\lecrr)\nonumber\\
       =& \log\tau(C_{1,r})+\log\tau(C_2)\label{eq:HA13}\\
       &+\log\tau(\lectt)+\log\cos\TT\; \chi(B).\nonumber
\end{align}
\end{theorem}
When comparing the last formula with Theorem \ref{thm:Theorem1}
one should note that for $\TT=\pi/4$ we have $\log\cos\TT=\log
\frac{1}{\sqrt{2}}=-\frac12 \log 2$.

\begin{proof} For $0<\TT,\TT'<\pi/2$ we have the cochain isomorphism
 (\textit{cf.} \Eqref{eq:ExactSeq1Comp})
\begin{equation}\label{eq:14}
 \phi_{\TT,\TT'}: \cTT\longrightarrow \cTTp,\quad (\xi_1,\xi_2)\mapsto
 (\xi_1,\frac{\tan\TT}{\tan\TT'} \xi_2),
\end{equation} 
hence by Lemma \ref{l:OMZ5}
\begin{multline}
 \log\tau(\cTT)=\log\tau(\cTTp)-
 \altsum \log\Det\bigl(\phi_{\TT,\TT'}\restriction H^j(\cTT)\bigr)\\
 +\altsum \log\Det\bigl(\phi_{\TT,\TT'}\restriction (\cTT)^j\bigr).
\end{multline}
Taking $\dThetaPrimeAtTheta$ yields \Eqref{eq:HA7}.

Next we look at the analogues of \Eqref{eq:ExactSeq1Comp}
and \Eqref{eq:ExactSeq1CompLongExactCoh}
\begin{equation}\label{eq:HA16}
\xymatrix{
0 \ar[r]  &C_{1,r} \ar[r]^-{\ga_\TT}\ar[d]^{\id} 
           &\cTT \ar[r]^-{\beta_\TT}\ar[d]^{\phi_{\TT,\TT'}}
     &C_2 \ar[r]\ar[d]^{\tilde\phi_{\TT,\TT'}} &  0\\
     0 \ar[r]  &C_{1,r} \ar[r]^-{\ga_{\TT'}} 
           &\cTTp \ar[r]^-{\beta_{\TT'}}  
     &C_2 \ar[r] &  0,
     }
\end{equation}
where $\tilde\phi_{\TT,\TT'}(\xi)=\frac{\tan\TT}{\tan\TT'}\xi$
and the corresponding isomorphism between the long exact
cohomology sequences
\begin{equation}\label{eq:HA17}
\xymatrix{
\ldots   H^k(C_{1,r}) \ar[r]^-{\ga_{\TT,*}}\ar[d]^{\id} 
           &H^k(\cTT) \ar[r]^-{\beta_{\TT,*}}\ar[d]^{\phi_{\TT,\TT',*}}
     &H^k(C_2) \ar[r]^-{\delta_\TT}\ar[d]^{\tilde\phi_{\TT,\TT',*}} 
     & H^{k+1}(C_{1,r})\ar[d]^{\id}\ldots\\
     \ldots   H^k(C_{1,r}) \ar[r]^-{\ga_{\TT',*}} 
           &H^k(\cTTp) \ar[r]^-{\beta_{\TT',*}}
     &H^k(C_2) \ar[r]^-{\delta_{\TT'}} 
     & H^{k+1}(C_{1,r})\ldots
     }
\end{equation}
Following the argument after \Eqref{eq:ExactSeq1CompLongExactCoh}
we find that
\begin{align}
 \dThetaPrimeAtTheta \log\Det\bigl(\phi_{\TT,\TT',*}^j\bigr)^2
 &= -2\frac{2}{\sin 2\TT} \HAtrbetatheta,\label{eq:HA18}\\
\dThetaPrimeAtTheta \log\Det\bigl(\tilde\phi_{\TT,\TT',*}^{j}\bigr)
&= -2\frac{2}{\sin 2\TT} \dim C_2^j\label{eq:HA19}
  \end{align}
  and hence with Lemma \ref{l:OMZ5} applied to \Eqref{eq:HA17}
we arrive at \Eqref{eq:HA8}.

The analogue of \Eqref{eq:ExactSeq3Comp} is 
\begin{equation}\label{eq:HA20}
\xymatrix{
0 \ar[r]  &C_{1,r}\oplus C_{2,r}
                \ar[r]^-{\gamma_1\oplus \gamma_2}\ar[d]^{\id\oplus
                \tilde\phi_{\TT,\TT'}} 
           &\cTT \ar[r]^-{r_\TT}\ar[d]^{\phi_{\TT,\TT'}}
           &B \ar[r]\ar[d]^{\psi_{\TT,\TT'}=\frac{\tan\TT}{\tan\TT'}\id} &  0\\
     0 \ar[r]  &C_{1,r}\oplus  C_{2,r} \ar[r]^-{\ga_{\TT'}} 
           &\cTTp \ar[r]^-{r_{\TT'}}  
     &B \ar[r] &  0.
     }
\end{equation}
We apply Lemma \plref{l:OMZ5} to the induced isomorphism of the long exact
cohomology sequences and find
\begin{multline}\label{eq:HA21}
 \log\tau(\lecrr)- \log\tau(\lecrrp)\\
 =-\altsum  \log\Det(\phi_{\TT,\TT',*}:H^j\to H^j)
   +\altsum \log\Det(\tilde\phi_{\TT,\TT',*}:H^j\to H^j)\\
   +\altsum \log\Det(\Psi_{\TT,\TT',*}:H^j\to H^j),
\end{multline}
where $H^j$ is shorthand for the respective cohomology groups.
Since $\tilde\phi_{\TT,\TT'}$ and $\phi_{\TT,\TT'}$ are multiplication
operators we have
\begin{align}
 \altsum \log\Det(\tilde\phi_{\TT,\TT',*}:H^j\to H^j)&= \chi(C_{2,r}) \,
 \log\frac{\tan\TT}{\tan\TT'},\label{eq:HA22}\\
 \altsum \log\Det(\Psi_{\TT,\TT',*}:H^j\to H^j) & = \chi(B)\,
 \log\frac{\sin\TT}{\sin\TT'},\label{eq:HA23}
\end{align}
and together with \Eqref{eq:HA18} we obtain
\begin{align} 
\frac{d}{d\TT}&\log \tau\bigl(\lecrr\bigr) =\nonumber\\
&  =\frac{2}{\sin 2\TT}\Bigl[- \altsum \HAtrbetatheta + \chi(C_{2,r})\Bigr]
-\frac{\cos\TT}{\sin\TT}\; \chi(B).\label{eq:HA24}
\end{align}
Taking into account $\chi(C_{2,r})=\chi(C_2)-\chi(B)$ 
(\textit{cf}. \Eqref{eq:HA3}) and
$\frac{\cos\TT}{\sin\TT}-\frac{2}{\sin2\TT}=-\tan\TT$ we find
\Eqref{eq:HA9}.

Next we apply Prop. \ref{p:OMZ6} to the exact sequence \Eqref{eq:HA4}
and get
\begin{align}
 \log\tau(\cTT)=& \log\tau(C_{1,r})+\log\tau(C_{2})\label{eq:HA25}\\
       &+\log\tau(\lectt)\nonumber\\
       &+\frac 12 \altsum \log\Det(\gb\gb^*:C_2^j\to C_2^j).\nonumber
\end{align}
Here we have used \Eqref{eq:OMZ26} and
that $\ga$ is a partial isometry and thus $\ga^*\ga=\id$.
Analogously, we infer from \Eqref{eq:HA5}
\begin{align}
 \log\tau(\cTT)=& \log\tau(C_{1,r})+\log\tau(C_{2,r})\label{eq:HA26}\\
       &+\log\tau(\lecrr)\nonumber\\
       &+\frac 12 \altsum \log\Det(r_{\TT}r_{\TT}^*:B^j\to B^j).\nonumber
\end{align}
From \Eqref{eq:HA25} and \eqref{eq:HA26} one deduces the differentiability
statement \Eqref{eq:HA10}.

Finally we discuss the case that the maps $r_j, j=1,2$ are partial
isometries. Then for $(\xi_1,\xi_2)\in\cTT, \eta\in B$ we calculate
\begin{equation}
 \begin{split}
  \scalar{r_\TT(\xi_1}{\xi_2),b}&= \sint\cdot \scalar{r_1\xi_1}{b}+\cost\cdot
  \scalar{r_2\xi_2}{b}\\
  &= \scalar{(\xi_1,\xi_2)}{(\sint\cdot r_1^* b, \cost\cdot r_2^* b)}.
 \end{split}
\end{equation}
If $r_1$ and $r_2$ are partial isometries then 
$(\sint\cdot r_1^* b, \cost\cdot \, r_2^* b)\in\cTT$ and hence it
equals $r_\TT^*(b)$. Consequently $r_\TT r_\TT^* b=(\sin^2\TT+\cos^2\TT) b=b$
and thus $\Det(r_\TT r_\TT^*:B^j\to B^j)=1$. 
Therefore \Eqref{eq:HA26} reduces to \Eqref{eq:HA12}. 

Similarly, one calculates
\begin{equation}
    \Det(\gb\gb^*:C_2^j\to C_2^j)= (1+\tan^2)^{-\dim B^j},
\end{equation}
then \Eqref{eq:HA13} follows from \Eqref{eq:HA25}.
\end{proof}


\bibliography{mlbib.bib,localbib.bib} 

\def\cprime{$'$}
\providecommand{\bysame}{\leavevmode\hbox to3em{\hrulefill}\thinspace}
\providecommand{\MR}{\relax\ifhmode\unskip\space\fi MR }
\providecommand{\MRhref}[2]{%
  \href{http://www.ams.org/mathscinet-getitem?mr=#1}{#2}
}
\providecommand{\href}[2]{#2}
\begin{thebibliography}{\textsc{DoWo91}}

\bibitem[\textsc{APS75}]{APS:SARI}
\textsc{M.~F. Atiyah}, \textsc{V.~K. Patodi}, and \textsc{I.~M. Singer},
  \emph{Spectral asymmetry and {R}iemannian geometry. {I}}, Math. Proc.
  Cambridge Philos. Soc. \textbf{77} (1975), 43--69. \MR{0397797 (53 \#1655a)}

\bibitem[\textsc{BFK99}]{BurFriKap:TMB}
\textsc{D.~Burghelea}, \textsc{L.~Friedlander}, and \textsc{T.~Kappeler},
  \emph{Torsions for manifolds with boundary and glueing formulas}, Math.
  Nachr. \textbf{208} (1999), 31--91. \texttt{arXiv:9510010v1 [dg-ga]},
  \MR{1719799 (2000m:58053)}

\bibitem[\textsc{Boh09}]{Boh:RIF}
\textsc{M.~Bohn}, \emph{On rho invariants of fiber bundles}, PhD thesis, 263
  pages, Bonn, 2009. \texttt{arXiv:0907.3530v1 [math.GT]}

\bibitem[\textsc{BrLe92}]{BruLes:HC}
\textsc{J.~Br{\"u}ning} and \textsc{M.~Lesch}, \emph{Hilbert complexes}, J.
  Funct. Anal. \textbf{108} (1992), no.~1, 88--132. \MR{1174159 (93k:58208)}

\bibitem[\textsc{BrLe99}]{BruLes:EIC}
\textsc{J.~Br{\"u}ning} and \textsc{M.~Lesch}, \emph{On the {$\eta$}-invariant
  of certain nonlocal boundary value problems}, Duke Math. J. \textbf{96}
  (1999), no.~2, 425--468. \texttt{arXiv:9609001 [dg-ga,math.DG]}, \MR{1666570
  (99m:58180)}

\bibitem[\textsc{BrLe01}]{BruLes:BVPI}
\bysame, \emph{On boundary value problems for {D}irac type operators. {I}.
  {R}egularity and self-adjointness}, J. Funct. Anal. \textbf{185} (2001),
  no.~1, 1--62. \texttt{arXiv:9905181 [math.FA]}, \MR{1853751 (2002g:58034)}

\bibitem[\textsc{BrMa}]{BruMa:GFA}
\textsc{J.~Br{\"u}ning} and \textsc{X.~Ma}, \emph{On the gluing formula for the
  analytic torsion}, Preprint, downloaded from
  http://www.math.jussieu.fr/$\sim$ma/publi.html Nov 16, 2011.

\bibitem[\textsc{BrMa06}]{BruMa:AFR}
\bysame, \emph{An anomaly formula for {R}ay-{S}inger metrics on manifolds with
  boundary}, Geom. Funct. Anal. \textbf{16} (2006), no.~4, 767--837.
  \MR{2255381 (2007i:58042)}

\bibitem[\textsc{BrSe87}]{BruSee:RES}
\textsc{J.~Br{\"u}ning} and \textsc{R.~Seeley}, \emph{The resolvent expansion
  for second order regular singular operators}, J. Funct. Anal. \textbf{73}
  (1987), no.~2, 369--429. \MR{899656 (88g:35151)}

\bibitem[\textsc{BrSe88}]{BruSee:ITF}
\bysame, \emph{An index theorem for first order regular singular operators},
  Amer. J. Math. \textbf{110} (1988), no.~4, 659--714. \MR{955293 (89k:58271)}

\bibitem[\textsc{Che79a}]{Che:ATH}
\textsc{J.~Cheeger}, \emph{Analytic torsion and the heat equation}, Ann. of
  Math. (2) \textbf{109} (1979), no.~2, 259--322. \MR{528965 (80j:58065a)}

\bibitem[\textsc{Che79b}]{Che:SGS}
\bysame, \emph{On the spectral geometry of spaces with cone-like
  singularities}, Proc. Nat. Acad. Sci. U.S.A. \textbf{76} (1979), no.~5,
  2103--2106. \MR{530173 (80k:58098)}

\bibitem[\textsc{Che83}]{Che:SGSR}
\bysame, \emph{Spectral geometry of singular {R}iemannian spaces}, J.
  Differential Geom. \textbf{18} (1983), no.~4, 575--657 (1984). \MR{730920
  (85d:58083)}

\bibitem[\textsc{CoMo95}]{ConMos:LIF}
\textsc{A.~Connes} and \textsc{H.~Moscovici}, \emph{The local index formula in
  noncommutative geometry}, Geom. Funct. Anal. \textbf{5} (1995), no.~2,
  174--243. \MR{1334867 (96e:58149)}

\bibitem[\textsc{DaFr94}]{DaiFre:EID}
\textsc{X.~Dai} and \textsc{D.~S. Freed}, \emph{{$\eta$}-invariants and
  determinant lines}, J. Math. Phys. \textbf{35} (1994), no.~10, 5155--5194,
  Topology and physics. \MR{1295462 (96a:58204)}

\bibitem[\textsc{Dar87}]{Dar:IRT}
\textsc{A.~Dar}, \emph{Intersection {$R$}-torsion and analytic torsion for
  pseudomanifolds}, Math. Z. \textbf{194} (1987), no.~2, 193--216. \MR{876230
  (88b:58132)}

\bibitem[\textsc{DoWo91}]{DouWoj:ALE}
\textsc{R.~G. Douglas} and \textsc{K.~P. Wojciechowski}, \emph{Adiabatic limits
  of the {$\eta$}-invariants. {T}he odd-dimensional {A}tiyah-{P}atodi-{S}inger
  problem}, Comm. Math. Phys. \textbf{142} (1991), no.~1, 139--168. \MR{1137777
  (92j:58110)}

\bibitem[\textsc{Gil95}]{Gil:ITH}
\textsc{P.~B. Gilkey}, \emph{Invariance theory, the heat equation, and the
  {A}tiyah-{S}inger index theorem}, second ed., Studies in Advanced
  Mathematics, CRC Press, Boca Raton, FL, 1995. \MR{1396308 (98b:58156)}

\bibitem[\textsc{Gru99}]{Gru:TEP}
\textsc{G.~Grubb}, \emph{Trace expansions for pseudodifferential boundary
  problems for {D}irac-type operators and more general systems}, Ark. Mat.
  \textbf{37} (1999), no.~1, 45--86. \MR{1673426 (2000c:35265)}

\bibitem[\textsc{HaSp10}]{HarSpr:ECM}
\textsc{L.~Hartmann} and \textsc{M.~Spreafico}, \emph{An extension of the
  {C}heeger-{M}\"uller {T}heorem for a cone},  \texttt{arXiv:1008.2987
  [math.DG]}.

\bibitem[\textsc{Les97}]{Les:OFT}
\textsc{M.~Lesch}, \emph{Operators of {F}uchs type, conical singularities, and
  asymptotic methods}, Teubner-Texte zur Mathematik [Teubner Texts in
  Mathematics], vol. 136, B. G. Teubner Verlagsgesellschaft mbH, Stuttgart,
  1997. \texttt{arXiv:dg-ga/9607005v1}, \MR{1449639 (98d:58174)}

\bibitem[\textsc{LeWo96}]{LesWoj:EIG}
\textsc{M.~Lesch} and \textsc{K.~P. Wojciechowski}, \emph{On the
  {$\eta$}-invariant of generalized {A}tiyah-{P}atodi-{S}inger boundary value
  problems}, Illinois J. Math. \textbf{40} (1996), no.~1, 30--46. \MR{1386311
  (97d:58194)}

\bibitem[\textsc{LiMa72}]{LioMag:NHB}
\textsc{J.-L. Lions} and \textsc{E.~Magenes}, \emph{Non-homogeneous boundary
  value problems and applications. {V}ol. {II}}, Springer-Verlag, New York,
  1972, Translated from the French by P. Kenneth, Die Grundlehren der
  mathematischen Wissenschaften, Band 182. \MR{0350178 (50 \#2671)}

\bibitem[\textsc{MaVe11}]{MazVer:ATM}
\textsc{R.~Mazzeo} and \textsc{B.~Vertman}, \emph{Analytic torsion on manifolds
  with edges},  \texttt{arXiv:1103.0448v1 [math.SP]}.

\bibitem[\textsc{Maz91}]{Maz:ETD}
\textsc{R.~Mazzeo}, \emph{Elliptic theory of differential edge operators. {I}},
  Comm. Partial Differential Equations \textbf{16} (1991), no.~10, 1615--1664.
  \MR{1133743 (93d:58152)}

\bibitem[\textsc{Mel93}]{Mel:APS}
\textsc{R.~B. Melrose}, \emph{The {A}tiyah-{P}atodi-{S}inger index theorem},
  Research Notes in Mathematics, vol.~4, A K Peters Ltd., Wellesley, MA, 1993.
  \MR{1348401 (96g:58180)}

\bibitem[\textsc{Mil66}]{Mil:WT}
\textsc{J.~Milnor}, \emph{Whitehead torsion}, Bull. Amer. Math. Soc.
  \textbf{72} (1966), 358--426. \MR{0196736 (33 \#4922)}

\bibitem[\textsc{M{\"u}l78}]{Mul:ATT}
\textsc{W.~M{\"u}ller}, \emph{Analytic torsion and {$R$}-torsion of
  {R}iemannian manifolds}, Adv. in Math. \textbf{28} (1978), no.~3, 233--305.
  \MR{498252 (80j:58065b)}

\bibitem[\textsc{M{\"u}l83}]{Mul:STR}
\bysame, \emph{Spectral theory for {R}iemannian manifolds with cusps and a
  related trace formula}, Math. Nachr. \textbf{111} (1983), 197--288.
  \MR{725778 (85i:58121)}

\bibitem[\textsc{M{\"u}l92}]{Mul:SGS}
\bysame, \emph{Spectral geometry and scattering theory for certain complete
  surfaces of finite volume}, Invent. Math. \textbf{109} (1992), no.~2,
  265--305. \MR{1172692 (93g:58151)}

\bibitem[\textsc{M{\"u}l93}]{Mul:ATR}
\bysame, \emph{Analytic torsion and {$R$}-torsion for unimodular
  representations}, J. Amer. Math. Soc. \textbf{6} (1993), no.~3, 721--753.
  \MR{1189689 (93m:58119)}

\bibitem[\textsc{MuVe11}]{MulVer:MAA}
\textsc{W.~M{\"u}ller} and \textsc{B.~Vertman}, \emph{The metric anomaly of
  analytic torsion on manifolds with conical singularities},
  \texttt{arXiv:1004.2067v3 2011-07-12 [math.SP]}.

\bibitem[\textsc{Paq82}]{Paq:PMP}
\textsc{L.~Paquet}, \emph{Probl\`emes mixtes pour le syst\`eme de {M}axwell},
  Ann. Fac. Sci. Toulouse Math. (5) \textbf{4} (1982), no.~2, 103--141.
  \MR{687546 (84e:58075)}

\bibitem[\textsc{RaSi71}]{RaySin:RTL}
\textsc{D.~B. Ray} and \textsc{I.~M. Singer}, \emph{{$R$}-torsion and the
  {L}aplacian on {R}iemannian manifolds}, Advances in Math. \textbf{7} (1971),
  145--210. \MR{0295381 (45 \#4447)}

\bibitem[\textsc{Sch91}]{Sch:PDO}
\textsc{B.-W. Schulze}, \emph{Pseudo-differential operators on manifolds with
  singularities}, Studies in Mathematics and its Applications, vol.~24,
  North-Holland Publishing Co., Amsterdam, 1991. \MR{1142574 (93b:47109)}

\bibitem[\textsc{Shu01}]{Shu:POS}
\textsc{M.~A. Shubin}, \emph{Pseudodifferential operators and spectral theory},
  second ed., Springer-Verlag, Berlin, 2001, Translated from the 1978 Russian
  original by Stig I. Andersson. \MR{1852334 (2002d:47073)}

\bibitem[\textsc{Tay96}]{Tay:PDEI}
\textsc{M.~E. Taylor}, \emph{Partial differential equations. {I}}, Applied
  Mathematical Sciences, vol. 115, Springer-Verlag, New York, 1996, Basic
  theory. \MR{1395148 (98b:35002b)}

\bibitem[\textsc{Ver09}]{Ver:ATB}
\textsc{B.~Vertman}, \emph{Analytic torsion of a bounded generalized cone},
  Comm. Math. Phys. \textbf{290} (2009), no.~3, 813--860.
  \texttt{arXiv:0808.0449v2 [math.DG]}, \MR{2525641 (2010d:58032)}

\bibitem[\textsc{Vis95}]{Vis:GRS}
\textsc{S.~M. Vishik}, \emph{Generalized {R}ay-{S}inger conjecture. {I}. {A}
  manifold with a smooth boundary}, Comm. Math. Phys. \textbf{167} (1995),
  no.~1, 1--102. \texttt{arXiv:hep-th/9305184v1}, \MR{1316501 (96f:58184)}

\end{thebibliography}
\bibliographystyle{amsalpha-lmp}
\end{document}